\documentclass[12pt]{amsart}
\usepackage[usenames,dvipsnames]{color}
\usepackage[colorlinks=true,pagebackref,hyperindex,citecolor=green,linkcolor=red]{hyperref}
\usepackage{fullpage}
\usepackage{setspace}
\usepackage{latexsym}
\usepackage{amsmath}
\usepackage{amsfonts}
\usepackage{amssymb}
\usepackage{mathrsfs}
\usepackage[all]{xy}
\usepackage{graphicx}
\usepackage{pst-3dplot}

\usepackage{mdwlist} 

  \theoremstyle{definition}
 \newtheorem{Theorem}{Theorem}[section]
 \newtheorem{Corollary}[Theorem]{Corollary}
 \newtheorem{Lemma}[Theorem]{Lemma}
 \newtheorem{Proposition}[Theorem]{Proposition}

\newtheorem{UHypothesis}[Theorem]{Universal Hypothesis}

\newtheorem{Main Theorem}[Theorem]{Main Theorem}

\newtheorem{Definition}[Theorem]{Definition}

\newtheorem{Remark}[Theorem]{Remark}

\newtheorem{Example}[Theorem]{Example}
\newtheorem{Notation}[Theorem]{Notation}
\newtheorem{Terminology}[Theorem]{Terminology}
\newtheorem{Discussion}[Theorem]{Discussion}
\numberwithin{equation}{subsection}

\newcommand{\up}[1]{\left\lceil #1 \right\rceil}

\newcommand{\mfrak}[1]{\mathfrak{#1}}

\newcommand{\tens}{\otimes}

\newcommand{\Spec}{\operatorname{Spec}}
\newcommand{\mSpec}{\operatorname{mSpec}}

\newcommand{\Support}{\operatorname{Supp}}

\newcommand{\Z}{\mathbb{Z}}

\renewcommand{\P}{\boldsymbol{P}}
\newcommand{\Pmax}{\bold{P}_{\text{max}}}

\renewcommand{\k}{\boldsymbol{k}}
\newcommand{\s}{\boldsymbol{s}}

\newcommand{\uu}{\boldsymbol{u}}
\renewcommand{\t}{\boldsymbol{t}}
\newcommand{\x}{\boldsymbol{x}}

\newcommand{\n}{\boldsymbol{\nu}}
\newcommand{\m}{\mfrak{m}}
\renewcommand{\a}{\mfrak{a}}

\newcommand{\0}{\boldsymbol{0}}
\renewcommand{\l}{\ell}

\newcommand{\LLambda}{\boldsymbol{\Lambda}}

\renewcommand{\1}{\textbf{1}}

\renewcommand{\bold}[1]{\mathchoice{\hbox{\boldmath $\displaystyle #1$}}
        {\hbox{\boldmath $\textstyle #1$}}
        {\hbox{\boldmath $\scriptstyle #1$}}
        {\hbox{\boldmath $\scriptscriptstyle #1$}}}

\newcommand{\lfpt}[2]{\bold{\operatorname{fpt}}_{#1}(#2)}
\newcommand{\fpt}[1]{\bold{\operatorname{fpt}}_{\m}(#1)}
\newcommand{\llct}[1]{\bold{\operatorname{lct}}_{\0}(#1)}

\newcommand{\lct}[1]{\bold{\operatorname{lct}}_{\0}(#1)}

\newcommand{\tr}[2]{\left \langle {#1} \right \rangle_{#2}}

\newcommand{\bracket}[2]{ {#1}^{[p^{#2}]}}

\newcommand{\characteristic}{\operatorname{char}}

\newcommand{\set}[1]{ \left\{ \, #1 \, \right\}}
\newcommand{\pair}[3]{\left( {#1}, {#3} \bullet {#2} \right)}

\renewcommand{\aa}{\bold{a}}
\renewcommand{\(}{\left(}
\renewcommand{\)}{\right)}

\renewcommand{\th}{\text{th}}

\newcommand{\new}[2]{ \nu_{#1} \( p^{#2} \)}

\newcommand{\vone}{\bold{1}}
\newcommand{\vleq}{\preccurlyeq}
\newcommand{\vl}{\prec}

\newcommand{\vgeq}{\succcurlyeq}

\newcommand{\sM}{\mathscr{M}} 
\newcommand{\EM}{\bold{\operatorname{E}}} 
\renewcommand{\vee}{\bold{\operatorname{v}}}

\newcommand{\digit}[2]{ {#1}^{(#2)}}
\newcommand{\convex}[1]{ {#1}^{\operatorname{convex}}}
\newcommand{\cone}[1]{{#1}^{\operatorname{cone}}}
\newcommand{\Newton}{\bold{N}}
\newcommand{\Nsigma}{\bold{\Lambda}}

\newcommand{\sC}{\mathscr{C}}
\newcommand{\sP}{\mathscr{P}}

\newcommand{\bv}{\bold{\operatorname{v}}}
\newcommand{\bw}{\bold{\operatorname{w}}}
\newcommand{\be}{\bold{\operatorname{e}}}

\newcommand{\bb}{\bold{\operatorname{b}}}

\newcommand{\bkappa}{\bold{\kappa}}
\newcommand{\eeta}{\bold{\eta}}

\newcommand{\LL}{\bold{\Lambda}}
\renewcommand{\ll}{\bold{\lambda}}
\newcommand{\fL}{f_{\LL}}
\newcommand{\GG}{\bold{\Gamma}}

\renewcommand{\L}{\mathbb{L}}

\begin{document}

\title
 {$F$-purity versus log canonicity for polynomials}
\author{ Daniel J. Hern\'andez }
\thanks{The author was partially supported by the National Science Foundation RTG grant number 0502170 at the University of Michigan.}

\begin{abstract} 
In this article, we consider the conjectured relationship between $F$-purity and log canonicity for polynomials over $\mathbb{C}$.  We associate to a collection $\sM$ of $n$ monomials a rational polytope $\P$ contained in $[0,1]^n$.  Using $\P$ and the Newton polyhedron associated to $\sM$, we define a non-degeneracy condition under which log canonicity and dense $F$-pure type are equivalent for all $\mathbb{C}^{\ast}$-linear combinations of the monomials in $\sM$. We also show that log canonicity corresponds to $F$-purity for very general polynomials.  Our methods rely on showing that the $F$-pure and log canonical threshold agree for infinitely primes, and we accomplish this by comparing these thresholds with the thresholds associated to their monomial ideals.  
\end{abstract}

\maketitle


\section*{Introduction}

Let $f \in \L[x_1, \cdots, x_m]$ be a polynomial over a field of characteristic $p>0$ with $f(\0)=0$.  Let $\bracket{\m}{e}: = \( x_1^{p^e}, \cdots, x_m^{p^e} \)$ denote the $e^{\th}$ Frobenius power of $\m := (x_1, \cdots, x_m)$.  The limit \begin{equation} \label{fptIntroD: e} \fpt{f} = \lim_{e \to \infty} \frac{ \max \set{ r : f^r \notin \bracket{\m}{e}}}{p^e} \end{equation} exists, and is called the \emph{$F$-pure threshold} of $f$ at $\m$ \cite{HW2002,TW2004, MTW2005}.  This invariant measures of the singularities of $f$ near $\0$, and is closely related to the theory of \emph{$F$-purity} and \emph{tight closure} \cite{HR1976, HH1990, HY2003}.  If $\sM$ is a collection of monomials, one may also define the $F$-pure threshold of the ideal generated by $\sM$, denoted $\fpt{\sM}$, and if $f$ is a $\L^{\ast}$-linear combination of the elements of $\sM$, one has the inequality \begin{equation} \label{intro: e} \fpt{f} \leq \min \set{ 1, \fpt{\sM} }. \end{equation} 

\noindent This inequality may be strict:  $\lfpt{\m}{u \cdot x^p + v \cdot y^p} = \frac{1}{p}$, while $\lfpt{\m}{x^p,y^p} = \frac{2}{p}$.

Suppose that $\sM$ consists of $n$ monomials in $m$ variables.  In Definition \ref{SplittingPolytope: D}, we associate to $\sM$ a rational polytope $\P$ contained in $[0,1]^n$ called the \emph{splitting polytope} of $\sM$.  The polytope $\P$ is closely related to the familiar \emph{Newton polyhedron} $\Newton \subseteq \mathbb{R}^{m}$ associated to $\sM$ and the geometry of each determines the value of $\fpt{\sM}$.  We call $\eeta \in \P$  \emph{maximal} if the sum of the entries of $\eeta$ is maximal among all coordinate sums of elements of $\P$ (see Definition \ref{MaxPoints: D}).  In this article, we identify conditions on $\Newton ,\P$ and $p$ under which equality holds in \eqref{intro: e}.  We summarize some of these results below.  In what follows, we assume that $f \in \L[x_1, \cdots, x_m]$ and is a $K^{\ast}$-linear combination of the monomials in $\sM$.

\begin{enumerate}
\item \label{pt1} If $\P$ contains a unique maximal point $\eeta$, then equality holds in \eqref{intro: e} whenever $(p-1) \cdot \eeta \in \mathbb{N}^n$.  This is a direct corollary of the Theorem \ref{MainTheorem}, which also produces a lower bound for $\fpt{f}$ when equality does not hold \eqref{intro: e}.
\item  \label{pt2} Suppose that $\min \set{ 1, \fpt{\sM} } = \fpt{\sM}$.  If $(p-1) \cdot \fpt{\sM} \in \mathbb{N}$, Proposition \ref{GeneralFixedp: P} states there exists a closed set $Z \subseteq \mathbb{A}^n_K$ such that equality holds in \eqref{intro: e} whenever the coefficients of $f$ are not in $Z$.  The condition $(p-1) \cdot \fpt{\sM} \in \mathbb{N}$ is necessary, as we saw in the example following \eqref{intro: e}. 
\item \label{pt4} Suppose instead that $\min \set{1, \fpt{\sM}} = 1$.  In the two dimensional case,  the failure of equality to hold in \eqref{intro: e} allows one to produce an explicit upper bound on $p$.  This statement is a corollary of Lemma \ref{Fedderfpt<1: L}, which gives the same conclusion in higher dimensions modulo some additional hypotheses and generalizes \cite[Lemma 2.3]{Fed1983}.
\suspend{enumerate}

We now switch gears and consider an invariant of singularities for polynomials defined over $\mathbb{C}$.  Let $f$ be a polynomial over $\mathbb{C}$ with $f(\0)=0$.  If $\lambda > 0$, the function $\frac{1}{|f|^{ \lambda}}$ has a pole at $\0$, and understanding how ``bad'' this pole is provides a measure of the singularities of $f$ at $\0$.  In particular, one may ask whether this function is $L^2$, which leads us to the definition of the \emph{log canonical threshold of $f$}, denoted $\lct{f}$: \[ \lct{f} =  \sup \set{ {\lambda} : \frac{1}{| f |^{2 \lambda}} \text{ is locally integrable at } \0}.\]

Log canonical thresholds can also be defined using information obtained via (log) resolution of singularities, and this invariant plays an important role in higher dimensional birational geometry \cite{BL2004, Laz2004}.  Remarkably, $F$-pure thresholds can be thought of as the positive characteristic analog of log canonical thresholds. \cite{Karen2000,HW2002,HY2003,Tak2004}.  We now briefly sketch the relationship between these two invariants.  If $f$ has rational coefficients, one may reduce them modulo $p$ to obtain polynomials $f_p$ over the finite fields $\mathbb{F}_p$ for $p \gg 0$.  Otherwise, one still obtains a \emph{family of positive characteristic models} $f_p$ over finite fields of characteristic $p$  via the process of \emph{reduction to positive characteristic}.  Using the results of \cite{HY2003}, it is observed in \cite{MTW2005} that 
\begin{equation} \label{IntroLimit: e} \fpt{f_p} \leq \llct{f} \text{ for } p \gg 0, \text{ and that } \lim_{p \to \infty} \lfpt{\m}{f_p} = \llct{f}.\end{equation} 

\begin{Example}
\label{CuspExample: E}
If $f = x^2 + y^3$, then $\llct{f} = \frac{5}{6}$, and 
\[ \lfpt{\m}{f_p} = \begin{cases} 1/2 & p=2 \\ 2/3 & p=3 \\ 5/6 & p \equiv 1 \bmod 6 \\ \frac{5}{6}-\frac{1}{6p} & p \equiv 5 \bmod 6 \end{cases}. \]  
\end{Example}

We say that \emph{log canonicity equals dense $F$-purity for $f$} whenever $\fpt{f_p} = \lct{f}$ for infinitely many $p$; see Remark \ref{TerminologyJustification: R} for a justification of this terminology.  Note that log canonicity equals dense $F$-purity for $x^3 + y^2$, as $p \equiv 1 \bmod 6$ for infinitely many primes.

\begin{Example}  Let $f \in \mathbb{Q}[x,y,z]$ be a form of degree $3$ with isolated singuarity at $\0$, so that $f$ defines an elliptic curve $E \subseteq \mathbb{P}^2$.  Then, $\lct{f} = 1$, and $\fpt{f_p} = 1$ if and only if $E_p = \mathbb{V} \( f_p \)$ is not \emph{supersingular}.  That log canonicity equals dense $F$-purity in this example follows from work of Serre.  \cite[Example 4.6]{MTW2005}.
\end{Example}

It is conjectured that log canonicity equals dense $F$-purity for all polynomials, and verifying this correspondence is a long-standing open problem \cite{Fed1983,Karen1997, EM2006}.  We now summarize the results in this article related to this correspondence.

\resume{enumerate}
\item \label{pt4} Let $\sM$ again denote a collection of monomials, and suppose that its associated polytope $\P$ contains a unique maximal point.  Theorem \ref{UniquePointTheorem} then states that log canonicity equals dense $F$-purity for every $\mathbb{C}^{\ast}$-linear combination of the monomials of $\sM$.
\item  \label{pt5} In Theorem \ref{VeryGeneralTheorem}, we also see that log canonicity equals dense $F$-purity for all polynomials whose coefficients form an algebraically independent set over $\mathbb{Q}$.  
\end{enumerate}

We now briefly outline our methods for proving the statements in \eqref{pt4} and \eqref{pt5}. As with $F$-pure thresholds, one may extend the definition of the log canonical threshold to any ideal vanishing at $\0$.  If $\sM$ denotes a set of monomials and $f$ is a $\mathbb{C}^{\ast}$-linear combination of the members of $\sM$, then $\lct{f} \leq \min \set{ 1, \lct{\sM} }$.  The value of $\fpt{\sM}$ has a natural formula in terms of the geometry of the splitting polytope $\P$, while $\lct{\sM}$ can be computed via the Newton polyhedron $\Newton$.  On the other hand, the connection between these two polytopes allows one to conclude that $\fpt{\sM} = \lct{\sM}$, a formula that also follows from the general statements in \cite{HY2003}.  We continue to assume that $f$ is a $\mathbb{C}^{\ast}$-linear combination of the members of $\sM$.  Using the results referenced in \eqref{pt1} and \eqref{pt2}, we are able to show that $\fpt{f_p} = \min \set{1, \fpt{\sM}}$ for infinitely many $p$, and applying the relations in \eqref{IntroLimit: e} shows that for such $p$, 
\[ \fpt{f} \leq \lct{f} \leq \min \set{1, \lct{\sM} } = \min \set{ 1, \fpt{f} } = \fpt{f},\] forcing equality throughout.  

\subsection*{Acknowledgments} Many of the results contained in this article appear in my Ph.D. thesis, completed at the University of Michigan.  I would like to thank my advisor, Karen Smith, for all her guidance during this project.  I am also grateful to Emily Witt, both for her support and for her assistance with many of the details in this article.  I would also like to thank Mel Hochster for talking with me about $F$-purity, and Vic Reiner for answering some of my questions regarding convex geometry.

\section{Base $p$ expansions}

\begin{Definition}
\label{Expansion: D}  Let $\alpha \in (0,1]$.  A \emph{non-terminating base $p$ expansion of $\alpha$} is the unique expression of the form $\alpha = \sum_{e \geq 1} \frac{a_e}{p^e}$ with the property that the integers $a_e$ are in $[0, p-1]$ and are all not eventually zero.  The number $a_e$ is called the \emph{$e^{\text{\th}}$ digit  of $\alpha$ in base $p$}. 
\end{Definition}

For example, $1 = \sum_{e\geq 1} \frac{p-1}{p^e}$ is a non-terminating base $p$ expansion of $1$.

\begin{Definition}
\label{TruncationDefinition}

Let $\alpha \in (0,1]$, and fix a prime $p$.

\begin{enumerate}  
\item $\digit{\alpha}{e}$ will always denote the $e^{\th}$ digit of $\alpha$ in base $p$.  By convention, $\digit{\alpha}{0}  = \digit{0}{e} = 0$.
\item We call $\tr{\alpha}{d} : = \sum_{e=1}^{d} \frac{ \digit{\alpha}{e}}{p^e}$ the \emph{$e^{\th}$ truncation of $\alpha$ in base $p$}.
\item For $\bold{\alpha}=\left( \alpha_1, \cdots, \alpha_n \right) \in  [0,1]^n$, we set $\tr{\bold{\alpha}}{e}:= \left( \tr{\alpha_1}{e}, \cdots, \tr{\alpha_n}{e} \right)$.
\end{enumerate}
\end{Definition}

\begin{Terminology} 
When referring to the $e^{\th}$ digit (respectively, truncation) of $\alpha$, we will always mean the $e^{\th}$ digit (respectively, truncation) in some base $p$ which will always be obvious from the context (and will often be equal to the characteristic of an ambient field).  This explains the absence of the base prime $p$ in the notation $\digit{\alpha}{e}$ and $\tr{\alpha}{e}$.
\end{Terminology}

\begin{Lemma}
\label{SimplifiedTruncation: L}
If $(p^e -1) \cdot \alpha \in \mathbb{N}$ for some $e$, then $(p^e-1) \cdot \alpha = p^e \cdot \tr{\alpha}{e}$.  
\end{Lemma}
\begin{proof}
This follows from the observation (whose verification is left to the reader) that if $(p^e-1) \cdot \alpha \in \mathbb{N}$, then the digits of $\alpha$ (in base $p$) are periodic and repeat after $e$ terms.
\end{proof}

\begin{Lemma}
\label{TruncationLemma} If $\alpha \in (0,1]$, then $\tr{\alpha}{e} \in \frac{1}{p^e} \cdot \mathbb{N}$, and $\tr{\alpha}{e} < \alpha \leq \tr{\alpha}{e} + \frac{1}{p^e}$.  Furthermore, if $(p-1) \cdot \alpha \in \mathbb{N}$, then $\digit{\alpha}{e} = (p-1) \cdot \alpha$ for every $e \geq 1$.
\end{Lemma}

\begin{proof}  The first two assertions follow from the definitions.  For the third, note that the non-terminating base $p$ expansion for $\alpha$ can be obtained by multiplying each term in the non-terminating base $p$ expansion of $1 = \sum_{e \geq 1} \frac{p-1}{p^e}$ by $\alpha$.
\end{proof}

\begin{Definition}
\label{CarryingDefinition}
Let $(\alpha_1, \cdots, \alpha_n) \in [0,1]^n$. We say  \emph{the digits of $\alpha_1, \cdots, \alpha_n$ add without carrying} (in base $p$) if  $\digit{\alpha_1}{e} + \cdots + \digit{\alpha_n}{e} \leq p-1$ for every $e \geq 1$.
For $(k_1, \cdots, k_n) \in \mathbb{N}^n$, we say that the digits of $k_1, \cdots, k_n$ \emph{add without carrying} (in base $p$) if the analogous condition holds for the digits in the unique base $p$ expansions of the integers $k_1, \cdots, k_n$.
\end{Definition}

\begin{Remark}  We point out that the digits of $\alpha_1, \cdots, \alpha_n$ add without carrying if and only if the digits of the integers $p^e \tr{\alpha_1}{e}, \cdots, p^e \tr{\alpha_n}{e}$ add without carrying for every $e \geq 1$.
\end{Remark}

The concept of adding without carrying is relevant  in light of the following classical result.

\begin{Lemma}\cite{Dickson, Lucas} 
\label{LucasLemma}  
 Let $\(k_1, \cdots k_n \) \in \mathbb{N}^n$, and set $N= \sum_i k_i$.  Then, the multinomial coefficient $\binom{N}{\k} := \frac{ N! }{k_1 ! \cdots k_n!} \not \equiv 0 \mod p$ if and only if the digits of $k_1, \cdots, k_n$ add without carrying (in base $p$). 
\end{Lemma}

\begin{Lemma} 
\label{ConstantExpansionLemma}  Let $( \alpha_1, \cdots, \alpha_n) \in \mathbb{Q}^n \cap [0,1]^n$.
\begin{enumerate}
\item If $\alpha_1 + \cdots + \alpha_n \leq 1$,  there exist infinitely many primes $p$ for which the digits of $\alpha_1, \cdots, \alpha_n$ add without carrying (in base $p$).
\item Otherwise,  there exist infinitely many primes $p$ for which  $\digit{\alpha_1}{1} + \cdots + \digit{\alpha_n}{1} \geq p$.
\end{enumerate}
\end{Lemma}

\begin{proof} 
By Dirichlet's theorem on primes in arithmetic progressions, $(p-1) \cdot (\alpha_1, \cdots, \alpha_n) \in \mathbb{N}^n$ for infinitely many primes $p$.  For such primes, Lemma \ref{TruncationLemma} shows that $\digit{\alpha_i}{e} = (p-1) \cdot \alpha_i$ are the digits of $\alpha_i$.  The lemma follows from this observation.
\end{proof}

\section{$F$-pure thresholds}
\label{Purity: S}


\begin{UHypothesis}
A field $\L$ of prime characteristic $p > 0$ is called \emph{$F$-finite} if $[\L: \L^p] < \infty$.  Throughout this article, all fields of prime characteristic will be assumed to be $F$-finite. 
\end{UHypothesis}

Let $R = \L[x_1, \cdots, x_m]$ denote the polynomial ring over a field of characteristic $p>0$, and let $f$ be a non-zero polynomial in $R$ with $f(\0) = 0$, so that $f \in \m:=(x_1, \cdots, x_m).$  For every $I \subseteq R$, let $\bracket{I}{e}$ denote the ideal generated by the set $\set{ r^{p^e} : r \in I}$.  We call $\bracket{I}{e}$ the \emph{$e^{\text{th}}$ Frobenius power} of $I$.  Note that $ \bracket{\( \bracket{I}{e} \)}{\l} = \bracket{I}{e+\l}$.  It is well known that $R$ is a finitely-generated, free  $R^{p^e}$-module. Using this fact, it is easy to verify that $g^{p^e} \in \bracket{I}{e}$ if and only if $g \in I$ .  
Consider the values \begin{equation} \label{nu: deq} \new{f}{e} : = \max \set{ a : f^a \notin \bracket{\m}{e} }. \end{equation}
 
As $f \in \m$, we have that $f^{p^e} \in \bracket{\m}{e}$, so that $0 \leq \new{f}{e} \leq p^e-1$.  As $f^a \notin \bracket{\m}{e}$ implies that $\( f^{a} \)^{p^d} = f^{p^d a} \notin \bracket{\m}{e+d}$, we have that $p^d \cdot \new{f}{e} \leq \new{f}{e+d}$.  These observations show that the numbers $\frac{\new{f}{e}}{p^e}$ define a non-decreasing sequence of non-negative rational numbers contained in the unit interval.

\begin{Definition} \cite{TW2004, MTW2005}   The limit $\fpt{f}:= \lim \limits_{e \to \infty} \frac{\new{f}{e}}{p^e}$ exists, and is called the \emph{$F$-pure threshold of $f$ at $\m$}.
\end{Definition}

\begin{Remark}
\label{FTleq1Remark}
As $f \neq 0$, there exists $e \geq 1$ such that $f \notin \bracket{\m}{e}$.  Thus, $\new{f}{e} \geq 1$, and it follows that $\lfpt{\m}{f} \in (0,1]$.
\end{Remark}

\begin{Remark}
The $\new{f}{e}$ determine $\fpt{f}$ by definition, but it turns out that the converse is true as well.  Specifically, it is shown in \cite[Proposition 1.9]{MTW2005} that \begin{equation} \label{nurecovery: e} \new{f}{e} =  \up{p^e \cdot \fpt{f}} - 1, \end{equation} where $\up{ \cdot }$ denotes the least integer greater than function.  For a generalization, see \cite{Singularities}.
\end{Remark}

\begin{Lemma}
\label{OneConditionSplittingLemma}
Let $\lambda \in [0,1]$ be a rational number such that $(p^e-1) \cdot \lambda \in \mathbb{N}$ for some $e \geq 1$.  Then $f^{(p^e-1) \cdot \lambda } \notin \bracket{\m}{e}$ if and only if $\lfpt{\m}{f} \geq \lambda$.
\end{Lemma}

\noindent Lemma \ref{OneConditionSplittingLemma}  is closely related to results from \cite{Singularities} and {\cite{Schwede2008}}.

\begin{proof}[Proof of Lemma \ref{OneConditionSplittingLemma}] First, suppose that $\fpt{f} \geq \lambda$.  It follows from \eqref{nurecovery: e} that \[ \new{f}{e} \geq \up{ p^e \lambda } - 1 = \up{ (p^e-1) \cdot \lambda + \lambda} - 1 = (p^e-1) \cdot \lambda + \up{\lambda} - 1 = (p^e-1) \cdot \lambda,\] so that $f^{(p^e-1) \cdot \lambda}  \notin \bracket{\m}{e}$.   We now address the other implication.

As $\bracket{\m}{e}$ is $\m$-primary, we have that $\bracket{\m}{e} \cdot R_{\m} \cap R = \bracket{\m}{e}$.  Thus, we may localize at $\m$, and assume that $(R,\m)$ is a local ring.  Next, note that \begin{equation} \label{induction0: e} \( p^{e(d+1)} - 1 \) \cdot \lambda =  p^e \cdot \(p^{ed}-1\) \cdot \lambda + \( p^e- 1 \) \cdot \lambda, \end{equation} which shows that $(p^{ed} - 1) \cdot \lambda \in \mathbb{N}$ for all $d \geq 1$.  We now induce on $d$ to show that \begin{equation} \label{induction1: e} f^{(p^{ed} - 1) \cdot \lambda} \notin \bracket{\m}{ed} \text{ for every $d \geq 1$}, \end{equation}  the base case being our hypothesis.  By means of contradiction, assume that \eqref{induction1: e} holds, but that $f^{(p^{e(d+1)} - 1) \cdot \lambda} \in \bracket{\m}{e(d+1)}$.  Applying \eqref{induction0: e} shows that 
\begin{equation} \label{induction2: e} f^{(p^e-1) \cdot \lambda} \in \( \bracket{\m}{e(d+1)} : f^{p^e \( p^{ed} - 1\) \cdot \lambda} \) = \bracket{ \( \bracket{\m}{ed} : f^{(p^{ed}-1) \cdot \lambda} \) }{e} \subseteq \bracket{\m}{e}, \end{equation}

\noindent which is a direct contradiction of our initial hypothesis.  Thus,  \eqref{induction1: e} holds by induction, and implies that $\frac{\new{f}{ed}}{p^{ed}} \geq \frac{\( p^{ed}-1\) \cdot \lambda}{p^{ed}}$.  Finally, letting $d \to \infty$ shows that $\fpt{f} \geq \lambda$.
\end{proof}

We can generalize the above setup as follows:  Given any ideal $\a \subseteq \m$, let 
\begin{equation} \label{nuideal: deq} \new{\a}{e} := \max \set{ r : \a^r  \not \subseteq \bracket{\m}{e} } . \end{equation}

As before, $\frac{\new{\a}{e}}{p^e}$ defines a non-decreasing sequence of non-negative rational numbers.  If $\a$ is generated by $N$ elements, one can also check that $\new{\a}{e} \leq N \( p^e  - 1 \) + 1$, so that $\lim_{e \to \infty} \frac{\new{\a}{e}}{p^e}$ exists, and is bounded above by $N$.  We will call $\lim_{e \to \infty} \frac{\new{\a}{e}}{p^e}$ the $F$-pure threshold of $\a$ at $\m$, and denote it by $\fpt{\a}$.  As before, it is easy to see that $\fpt{\a} > 0$.

\begin{Notation}
We often write $\lfpt{\m}{f}$ and $\lfpt{\m}{\a}$ rather than $\lfpt{\m}{R,f}$ and $\lfpt{\m}{R,\a}$. Furthermore, if $\mathscr{N}$ is a subset of $R$, we will use $\fpt{\mathscr{N}}$ to denote the $F$-pure threshold of the ideal generated by $\mathscr{N}$.
\end{Notation}

\begin{Remark}
\label{RationalityRemark}  Note that if $f \in \a$, then $\fpt{f} \leq \fpt{a}$.  Thus, if $f$ is a $K^{\ast}$-linear combination of a set of monomials $\sM$, it follows from this observation and Remark \ref{FTleq1Remark} that $\fpt{f} \leq \min \set{1, \fpt{\sM}}$.  Also, although it is not obvious,  $\lfpt{\m}{f}$ and $\lfpt{\m}{\a}$ are rational numbers \cite[Theorem 3.1]{BMS2008}.
\end{Remark}



\section{Splitting polytopes and Newton polyhedra}

\subsection{On splitting polytopes}

\begin{Notation} For $\s = (s_1, \cdots, s_n) \in \mathbb{R}^n$, $|\s|$ will denote the coordinate sum $s_1 + \cdots + s_n$.  We stress that $|\cdot |$ is \emph{not} the usual Euclidean norm on $\mathbb{R}^n$.  Furthermore, when dealing with elements of $\mathbb{R}^n$, we use $\vl$ and $\vleq$ to denote component-wise (strict) inequality.  Finally, $\vone_m$ will denote the element $(1, \cdots, 1) \in \mathbb{R}^m$.
\end{Notation}

Let $\sM := \set{ \x^{\aa_1}, \cdots, \x^{\aa_n} }$ be a collection of distinct monomials in $x_1, \cdots, x_m$, where each exponent vector $\aa_i$ is a non-zero element of $\mathbb{N}^m$.  


\begin{Definition}
\label{SplittingMatrixDefinition}
We call the $m \times n$ matrix $\EM: = \left( \aa_1 \cdots \aa_n \right)$ the \emph{exponent matrix of $\sM$}.
\end{Definition}



\begin{Definition} 
\label{SplittingPolytope: D}
We call $\P: = \set{ \s \in \mathbb{R}_{\geq 0}^n  : \EM \s \vleq \vone_m }$ the \emph{splitting polytope of $\sM$}.  As $\EM$ has non-negative integer entries, $\P$ is contained in $[0,1]^n$.  \end{Definition}

\begin{Definition}
\label{MaxPoints: D}
For $\lambda \in \mathbb{R}_{\geq 0}$, let $\P_{\lambda}$ denote the hyperplane section  $\P \cap \set{ \s : | \s | = \lambda }$.  If $\alpha = \max \set{ | \s | : \s \in \P}$,  we set $\Pmax:=\P_{\alpha}$, and we call the elements of $\Pmax$ the \emph{maximal points} of $\P$.
\end{Definition}

\begin{Example}
\label{DiagonalPolytope: E}
If $\sM = \set{ x_1^{d_1}, \cdots, x_m^{d_m}}$, then $\P = \set{ \s \in \mathbb{R}^m : \0  \vleq \s \vleq \( \frac{1}{d_1}, \cdots, \frac{1}{d_m} \) }$, and $\Pmax = \set{ \left( \frac{1}{d_1}, \cdots, \frac{1}{d_m} \right)}$.
\end{Example}

\begin{Example}
\label{PyramidPolytope: E}
If  $\sM = \{ x^a, y^b, (xy)^c \}$, then $\P = \set{ \s \in \mathbb{R}_{\geq 0}^3 :  \begin{array}{c}  a s_1 + c s_3 \leq 1 \\ b s_2 + c s_3 \leq 1 \end{array} }$ is the convex hull of $\vee_1 = \left(\frac{1}{a},0,0\right), \vee_2 = \left(0,\frac{1}{b},0\right), \vee_3=\left( \frac{1}{a},\frac{1}{b},0\right)$,  $\vee_4 = \left(0,0,\frac{1}{c}\right)$, and $\0$. 
\end{Example}

\begin{figure}[h]
\begin{tabular}{ccc}
\psset{xunit=1.5in,yunit=1.5in}
\begin{pspicture}(0,-.3)(0,.65)
\psset{Alpha=60, Beta=20}
\psset{linewidth=.5pt}
\pstThreeDLine[linewidth=.5pt,arrows=->](0,0,0)(0,0,.6)
\pstThreeDLine[linewidth=.5pt,arrows=->](0,0,0)(.9,0,0)
\pstThreeDLine[linewidth=.5pt,arrows=->](0,0,0)(0,.7,0)

\pstThreeDPut(1,0,0){\small{$s_1$}}
\pstThreeDPut(0,.8,0){\small{$s_2$}}
\pstThreeDPut(0,0,.7){\small{$s_3$}}

\psset{dotscale=1}
\pstThreeDDot(0,0,0)
\pstThreeDDot(.6,0,0)
\pstThreeDDot(0,.5,0)
\pstThreeDDot(0,0,.4)
\pstThreeDDot(.6,.5,0)
\pstThreeDDot(.6,0,.4)
\pstThreeDDot(0,.5,.4)
\pstThreeDDot[dotscale=2.2](.6,.5,.4)
\psset{linewidth=.65pt}
\pstThreeDBox(0,0,0)(.6,0,0)(0,.5,0)(0,0,.4)
\pstThreeDLine[linestyle=dashed]{->}(.61,.51,.8)(.61,.51,.45)
\pstThreeDPut(.61,.51,.85){\tiny{$\Pmax$}}
\end{pspicture}
& \hspace{.3 \linewidth}  & 
\psset{xunit=1.5in,yunit=1.5in}
\begin{pspicture}(0,-.3)(0,.65)
\psset{Alpha=60, Beta=20}
\psset{linewidth=.5pt}
\pstThreeDLine[linewidth=.5pt,arrows=->](0,0,0)(0,0,.6)
\pstThreeDLine[linewidth=.5pt,arrows=->](0,0,0)(.9,0,0)
\pstThreeDLine[linewidth=.5pt,arrows=->](0,0,0)(0,.7,0)

\pstThreeDPut(1,0,0){\small{$s_1$}}
\pstThreeDPut(0,.8,0){\small{$s_2$}}
\pstThreeDPut(0,0,.7){\small{$s_3$}}

\psset{dotscale=1}
\pstThreeDDot(0,0,0)
\pstThreeDDot(.6,0,0)
\pstThreeDDot(0,.5,0)
\pstThreeDDot(0,0,.45)
\pstThreeDDot(.6,.5,0)
\psset{linewidth=.65pt}
\pstThreeDLine(.6,0,0)(0,0,.45)
\pstThreeDLine(0,.5,0)(0,0,.45)
\pstThreeDLine(.6,.5,0)(0,0,.45)
\pstThreeDLine(.6,0,0)(.6,.5,0)
\pstThreeDLine(.6,0,0)(.6,.5,0)
\pstThreeDLine(.6,0,0)(.6,.5,0)
\pstThreeDLine(0,.5,0)(.6,.5,0)
\pstThreeDLine(0,.5,0)(0,0,0)
\pstThreeDLine(.6,0,0)(0,0,0)
\pstThreeDLine(0,0,0)(0,0,.45)
\pstThreeDPut(.7,0,.1){\small{$\vee_1$}}
\pstThreeDPut(-.15,.5,0){\small{$\vee_2$}}
\pstThreeDPut(.85,.65,0){\small{$\vee_3$}}
\pstThreeDPut(0,.1,.5){\small{$\vee_4$}}
\end{pspicture}
\end{tabular}
\caption{The splitting polytopes from Examples \ref{DiagonalPolytope: E} and \ref{PyramidPolytope: E}.}
\end{figure}
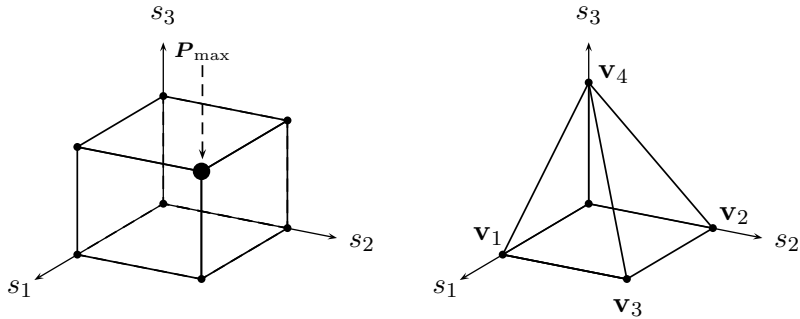

The following illustrates the connection between splitting polytopes and $F$-pure thresholds.

\begin{Proposition}
\label{Monomialfpt: P} Over any field of characteristic $p>0$,  $\fpt{\sM} = \max \set{| \s | : \s \in \P}$.
\end{Proposition}

\begin{proof} Note that the elements $\x^{k_1 \aa_1} \cdots \x^{k_n \aa_n}= \x^{\EM \k}$ with $k_1 + \cdots + k_n = N$ generate $(\sM)^N$.  Let $\new{}{e}: = \new{\( \sM \)}{e}$ be as in \eqref{nuideal: deq}.  As $\( \sM \)^{\new{}{e}} \nsubseteq \bracket{\m}{e}$, there exists $\k \in  \mathbb{N}^n$ with $| \k | = \new{}{e}$ such that $\x^{\EM \k} \notin \bracket{\m}{e}$.  By definition, $\frac{1}{p^e} \cdot \k \in \P$, and consequently $\frac{\new{}{e}}{p^e} = \frac{| \k |}{p^e} = \left | \frac{1}{p^e} \cdot \k \right |$, which by definition is bounded above by $\max \set{ |\s|: \s \in \P }$.  Taking $e \to \infty$ then shows that $\lfpt{\m}{\sM} \leq \max \set{|\s|: \s \in \P}$.  

Next, choose $\eeta \in \Pmax$.  As $\tr{\eeta}{e} \vl \eeta$, we have that $\EM \tr{\eeta}{e} \vl \EM \eeta \vleq \1_m$, and so $\x^{p^e \EM \tr{\eeta}{e}}$ is contained in $\( \sM \)^{p^e | \tr{\eeta}{e} |}$ but not  $\bracket{\m}{e}$. It follows that $\frac{\new{\a}{e}}{p^e} \geq | \tr{\eeta}{e} |$, and again letting $e \to \infty$ shows that $\fpt{\sM} \geq | \eeta | = \max \{ |\s| : \s \in \P \}$.  
\end{proof}

\subsection{On Newton polyhedra}
We next investigate the role of the Newton polyhedron associated to $\sM$, and begin by reviewing some basic notions from convex geometry.  

\begin{Discussion}
\label{convexgeo: d}
 Recall that $\sP \subseteq \mathbb{R}^m$ is called a \emph{polyhedral set} if there exist finitely many linear forms $L_1, \cdots L_d$ in $\mathbb{R}[z_1, \cdots, z_m]$ and elements $\beta_1, \cdots, \beta_d \in \mathbb{R}$ such that \begin{equation} \sP = \set{ \bv \in \mathbb{R}^m : L_i (\bv) \geq \beta_i \text{ for all } 1 \leq i \leq d}.\end{equation}
\noindent $\sP$ is called \emph{rational} if $L_i \in \mathbb{Q}[z_1, \cdots, z_m]$ and $\beta_i \in \mathbb{Q}$ for all $1 \leq i \leq d$.

Given any finite set $\sC \subseteq \mathbb{R}^m$, we use $\cone{\sC}$ to denote $\set{ \sum_{\bv \in \sC} \lambda_{\bv} \cdot \bv : \lambda_{\bv} \geq 0}$, the cone generated by $\sC$, and $\convex{\sC}$ to denote $\set{ \sum_{\bv \in \sC} \lambda_{\bv} \cdot \bv : \lambda_{\bv} \geq 0 \text{ and } \sum_{\bv \in \sC} \lambda_{\bv} = 1}$, the convex hull of $\sC$.  Though it is not obvious from these definitions, both $\cone{\sC}$ and $\convex{\sC}$ are polyhedral sets; see \cite[Theorem 4.1.1 and Theorem 3.2.5]{Webster}.  Additionally, these polyhedral sets are rational if and only if $\sC \subseteq \mathbb{Q}^m$.

If $L \in \mathbb{R}[z_1, \cdots, z_m]$ is a non-zero linear form and $\beta$ is a real number, we use $H^{\beta}_L$ to denote  $\set{ \bv \in \mathbb{R}^m : L(\bv) \geq \beta}$, the (upper) \emph{halfspace} determined by $L$ and $\beta$.  If $\sP$ is a polyhedral set, then $H^{\beta}_L$ is called a \emph{supporting} halfspace of $\sP$ if $\sP \subseteq H^{\beta}_L$ and $L^{-1} \( \beta \) \cap \sP \neq \emptyset$.  In this case,  $L^{-1} \( \beta \) \cap \sP$ is called an (exposed) \emph{face} of $\sP$.  Given two subsets $S$ and $S'$ of $\mathbb{R}^m$, we will use $S + S'$ to denote $\set{ \bv + \bw : \bv \in S, \bw \in S'}$, the Minkowski sum of $S$ and $S'$.  A set in $\mathbb{R}^m$ is polyhedral if and only if it is of the form $\convex{\sC} + \cone{\Gamma}$, where $\sC$ and $\Gamma$ are two finite (and possibly empty) subsets of $\mathbb{R}^m$ \cite[Theorem 4.1.2]{Webster}.
\end{Discussion}

In what follows, we abuse notation and use  $\convex{\sM}$ to denote $\convex{\set{ \aa_1, \cdots, \aa_n}}$, which has the following description in terms of the exponent matrix $\EM$ of $\sM$:  \begin{equation} \label{chulldescription: e} \convex{\sM} = \set{ \EM \s :  \s \in \mathbb{R}^n_{\geq 0} \text{ and } | \s | = 1 }. \end{equation}

\begin{Definition}
\label{Newton: D}
We call $\Newton: = \convex{\sM} + \mathbb{R}^m_{\geq 0}$ the \emph{Newton polyhedron} of $\sM$.
\end{Definition}

\noindent Note that $\Newton$ may also be described as the convex hull of $\set{ \bv : \x^{\bv} \in \( \sM \)}$, where $\( \sM \)$ denotes the monomial ideal generated by $\sM$.  Let $\be_1, \cdots, \be_m$ denote the standard basis for $\mathbb{R}^m$.  As $\mathbb{R}^m_{\geq 0} = \cone{ \set{\be_1, \cdots, \be_m}}$, it follows from Discussion \ref{convexgeo: d} that both $\P$ and $\Newton$ are rational polyhedral sets.  The following lemma gives some important conditions on the defining inequalities of $\Newton$.

\begin{Lemma}
\label{RationalCoefficients: L}
Let $L = \beta_1 z_1 + \cdots + \beta_m z_m \in \mathbb{Q}[z_1, \cdots, z_m]$ be a non-zero linear form, and suppose that $H^{\beta}_L$ is a supporting halfspace of $\Newton$.  Then $\beta_1, \cdots, \beta_m \geq 0$.  Furthermore,  $L^{-1} \( \beta \) \cap \Newton$ is bounded if and only if $\beta_1, \cdots, \beta_m > 0$.
\end{Lemma}

\begin{proof}
Fix $\bv \in L^{-1} \( \beta \) \cap \Newton$.  If $\lambda > 0$, then $\bv + \lambda \cdot \be_i \in \Newton$, and so 
\begin{equation} \label{ratcoeff: e} \beta \leq L \( \bv + \lambda \cdot \be_i \) = L\( \bv \) + \lambda \cdot L ( \be_i) = \beta + \lambda \cdot \beta_i.\end{equation}
\noindent As $\lambda > 0$, \eqref{ratcoeff: e} implies $\beta_i \geq 0$.  We now prove the contrapositive of the second assertion.  

The face $L^{-1}(\beta) \cap \Newton$ is unbounded if and only if some ray $\set{ \bv + \lambda \cdot e_i : \lambda \in \mathbb{R}_{>0}}$ is contained in $L^{-1}(\beta) \cap \Newton$.  However, this happens if and only if we have equality in \eqref{ratcoeff: e}, which then shows that $0 = \lambda \cdot \beta_i$.  As $\lambda > 0$, we conclude that $\beta_i=0$.  
\end{proof}

\begin{Remark}
In what follows, ``$\longleftrightarrow$ '' will be used to denote bijective correspondence between sets.  If $\lambda > 0$, it  follows from the definitions that 
\begin{align} \label{bijections: e} \P_{\lambda}  & \longleftrightarrow \frac{1}{\lambda} \cdot \P_{\lambda} = \set{ \s \in \mathbb{R}^n_{\geq 0} : | \s | = 1 \text{ and } \EM \s \vleq \frac{1}{\lambda} \cdot \vone_m } 
 \\  & \longleftrightarrow \set{ \( \s, \bold{w} \) \in \mathbb{R}^n_{\geq 0} \times \mathbb{R}^m_{\geq 0} : | \s | = 1 \text{ and } \EM \s + \bold{w} = \frac{1}{\lambda} \cdot \vone_m}. \notag \end{align}

It follows from \eqref{bijections: e} and \eqref{chulldescription: e} that $\P_{\lambda} \neq \emptyset \iff \frac{1}{\lambda} \cdot \vone_m \in \Newton$.  In particular, \begin{equation} \label{monomialidealequalities: e}  \fpt{\sM} =  \max \set{ | \s | : \s \in \P } = \max \set{ \lambda > 0 : \frac{1}{\lambda} \cdot \vone_m \in \Newton },  \end{equation}
where the first equality in \eqref{monomialidealequalities: e} holds by Proposition \ref{Monomialfpt: P}.
\end{Remark}

\subsection{Newton polyhedra in diagonal position, and $F$-pure thresholds}

\begin{Definition}
We will use $\alpha$ to denote the common values in \eqref{monomialidealequalities: e}.
\end{Definition}

\begin{Notation}
\label{alpha: N}
By definition, $\frac{1}{\alpha} \cdot \vone_m$ is contained in $\partial \Newton$, the boundary of $\Newton$.  By \cite[Theorem 3.2.2]{Webster}, $\partial \Newton$ is the union of the faces of $\Newton$, and we use $\Nsigma$ to denote the unique minimal face, with respect to inclusion, containing $\frac{1}{\alpha} \cdot \vone_m$.  We reorder $\sM$ and choose $1 \leq r \leq n$ so that $\aa_i \in \Nsigma$ if and only if $1 \leq i \leq r$.  We use $\sM_{\Nsigma}$ to denote the set $\set{ \x^{\aa_1}, \cdots, \x^{\aa_r}} \subseteq \sM$, and $\be_1, \cdots, \be_n$ to denote the standard basis of $\mathbb{R}^n$.
\end{Notation}

Suppose that $L = \beta_1 z_1 + \cdots + \beta_m z_m \in \mathbb{Q}[z_1, \cdots, z_m]$ is a non-zero linear form and $\beta$ is a rational number such that $H^{\beta}_L$ is a supporting halfspace of $\Newton$ and $\frac{1}{\alpha} \cdot \vone_m$ is on the face determined by $H^{\beta}_L$.  By Lemma \ref{RationalCoefficients: L}, $\beta_1, \cdots, \beta_m \geq 0$ and are not all zero, and so $\beta = L \( \frac{1}{\alpha} \cdot \vone_m \)  = \frac{1}{\alpha} \cdot \( \beta_1 + \cdots + \beta_m \) > 0$.  Normalizing, we obtain the following.

\begin{Definition}
There exists a unique linear form $L_{\Nsigma} \in \mathbb{Q}[z_1, \cdots, z_m]$ with non-negative coefficients such that $H^1_{L}$ is a supporting halfspace of $\Nsigma$ and $L_{\Nsigma}^{-1}(1) \cap \Newton = \Nsigma$.  We call $L_{\Nsigma}$ the \emph{linear form determined by $\Nsigma$.}
\end{Definition}

\begin{Definition}
\label{DiagonalPosition: D}
We say that $\Newton$ is in \emph{diagonal position} if $\Nsigma$ bounded.  By Lemma \ref{RationalCoefficients: L}, $\Newton$ is in diagonal position if and only if every $z_i$ appears in $L_{\Nsigma}$ with positive coefficient.
\end{Definition}

This terminology is motivated by the fact that if $\sM = \set{ x_1^{d_1}, \cdots, x_m^{d_m} }$, then $\sM = \sM_{\Nsigma}$ defines a Newton polyhedron in diagonal position.

\begin{Example}
The Newton polyhedra associated to the collection of monomials appearing in Example \ref{DiagonalPolytope: E} and \ref{PyramidPolytope: E} are in diagonal position.
\end{Example}

\begin{Lemma} 
\label{MaximalPointDP: L}
If $\Newton$ is in diagonal position, then $\s = (s_1, \cdots, s_n) \in \mathbb{R}^n_{\geq 0}$.  Then $\s \in \Pmax$ if and only if $s_{r+1} = \cdots = s_n = 0$ and $\EM \s = \vone_m$.
\end{Lemma}

\begin{proof}   
If $\s \in \Pmax$, set $\bkappa := \frac{1}{\alpha} \cdot \s$.  Note that $\bkappa \vgeq \0, | \bkappa | = 1$, and $\EM \bkappa \vleq \frac{1}{\alpha} \cdot \vone_m$.  By hypothesis, $L_{\Nsigma}$ has positive coefficients, and so preserves inequalities, and as $L_{\Nsigma} \equiv 1$ on $\Nsigma$, 
\begin{align}
\label{onemonomial: a}
1 = L_{\Nsigma} \( \frac{1}{\alpha} \cdot \vone_m \) \geq  L_{\Nsigma} \( \EM \bkappa \) =  L_{\Nsigma} \( \sum \limits_{i=1}^n \kappa_i \cdot \aa_i \) & =  \sum_{i=1}^r \kappa_i \cdot L_{\Nsigma}(\aa_i) + \sum_{i=r+1}^n \kappa_i \cdot L_{\Nsigma}(\aa_i) \\ 
 & = \sum_{i=1}^r \kappa_i + \sum_{i = r+1}^n \kappa_i \cdot L_{\Nsigma}(\aa_i). \notag
\end{align}

Substituting the equality $| \bkappa | =1$ into the initial term in \eqref{onemonomial: a}, we see that \begin{equation} \label{onemonomial: e} \sum_{i = r+1}^n \kappa_i \geq \sum_{i=r+1}^n \kappa_i \cdot L_{\Nsigma}(\aa_i).\end{equation}

\noindent By definition, $L_{\Nsigma}(\aa_i) > 1$ for $r+1 \leq i \leq n$, and so \eqref{onemonomial: e} is possible if and only if $\kappa_i$ (and hence $s_i$) is zero for $r+1 \leq i \leq n$. Furthermore, as the coefficients of $L_{\Nsigma}$ are positive, if any of the inequalities in $\EM \bkappa \vleq \frac{1}{\alpha} \cdot \vone_m$ were strict, then the inequality in \eqref{onemonomial: a} would be strict as well, which would imply that  $1 > \sum_{i=1}^r \kappa_i = | \bkappa | = 1$,  a contradiction.  We conclude that $\EM \bkappa = \frac{1}{\alpha} \cdot \vone_m$, and thus $\EM \s = \vone_m$.

Next, suppose that $\s = \sum_{i=1}^r s_i \cdot \be_i \vgeq \0$ and $\EM \s = 1$.  Thus, $\sum_{i=1}^r s_i  \cdot \aa_i = \EM \s = \vone_m$, and so $| \s | = \sum_{i=1}^r s_i = \sum_{i=1}^r s_i \cdot L_{\Nsigma}(\aa_i) = L_{\Nsigma} \(  \sum_{i=1}^r s_i \cdot \aa_i \) = L_{\Nsigma} \( \vone_m \) = \alpha \cdot L_{\Nsigma}\(\frac{1}{\alpha} \cdot \vone_m\) = \alpha$.
\end{proof}

\begin{Corollary}  
\label{Leftover: C}
There exists a maximal point $\eeta \in \Pmax$ with $(p^e-1) \cdot \eeta \in \mathbb{N}^n$ if and only if $(p^e-1) \cdot \alpha \in \mathbb{N}$ and $\( \sM \)^{(p^e-1) \cdot \alpha} \notin \bracket{\m}{e}$.  Furthermore, if either condition holds and $\Newton$ is in diagonal position,  then $\( \sM \)^{(p^e -1) \cdot \alpha}\equiv \( \sM_{\Nsigma} \)^{(p^e-1) \cdot \alpha} \equiv \( x_1^{p^e-1} \cdots x_m^{p^e-1} \)  \bmod \bracket{\m}{e}$. 
\end{Corollary}

\begin{proof}  Recall that $\( \sM \)^N$ is generated by the monomials $\x^{\EM \k}$ with $| \k | = N$.  If $\eeta \in \Pmax$ and $(p^e-1) \cdot \eeta \in \mathbb{N}^n$, then $(p^e-1) \cdot \alpha = (p^e-1) \cdot | \eeta | \in \mathbb{N}$, and $ \x^{(p^e-1) \cdot \EM \eeta}$ is in $\( \sM \)^{(p^e-1) \cdot \alpha}$, but not in $\bracket{\m}{e}$.  Conversely, if $\( \sM \)^{(p^e-1) \cdot \alpha} \not \subseteq \bracket{\m}{e}$, there exists $\k \in \mathbb{N}^n$ such that $| \k | = (p^e-1) \cdot \alpha$ and $\x^{\EM \k} \notin \bracket{\m}{e}$, so that $\frac{1}{p^e-1} \cdot \k \in \Pmax$.

We have just seen that the monomials in $\( \sM \)^{(p^e-1) \cdot \alpha}$ not contained in $\bracket{\m}{e}$ are of the form $\x^{\EM \k}$ for some index $\k$ satisfying $\frac{1}{p^e-1} \cdot \k \in \Pmax$.  By Lemma \ref{MaximalPointDP: L}, $\EM \k = (p^e-1) \cdot \vone_m$ and $k_i = 0$ for $r+1 \leq i \leq n$, so that $\x^{\EM \k} = \x^{(p^e-1) \cdot \vone_m}$ is in $\( \sM_{\Nsigma} \)^{(p^e-1) \cdot \alpha}$.
\end{proof}

\begin{Definition}  
\label{Support: D}
If $g = \sum_{\bold{I}} \uu_{\bold{I}} \cdot \x^{\bold{I}}$ is a polynomial over a field, we will use $\Support(g)$ to denote $\{ \x^{\bold{I}} : \uu_{\bold{I}} \neq 0 \}$, the set of \emph{supporting monomials} of $g$.  If $\sigma \subseteq \Support{g}$, we will use $g_{\sigma}$ to denote the polynomial $\sum_{\x^{\bold{I}} \in \sigma} u_{\bold{I}} \cdot \x^{\bold{I}}$, so that $\Support(g_{\sigma}) =\sigma$.
\end{Definition}

\begin{Proposition}
\label{SingleCoeff: P}
Let $f = \sum_{i=1}^n u_i \x^{\aa_i}$ be a polynomial over a field $\L$ of prime characteristic $p$ with $\Support(f) = \sM$.  If $\Newton$ is in diagonal position and  $(p^e - 1) \cdot \eeta \in \mathbb{N}^n$ for some $\eeta \in \Pmax$,  there exists a non-zero polynomial $\Theta_e \in \mathbb{Z}[t_1, \cdots, t_r]$ satisfying the following conditions:

\begin{enumerate}
\item $f^{(p^e-1) \cdot \alpha} \equiv \(f_{\Nsigma}\)^{(p^e-1) \cdot \alpha} \equiv \Theta_e (u_1, \cdots, u_r)\cdot x_1^{p^e-1} \cdots x_m^{p^e-1} \bmod \bracket{\m}{e}$. 
\item $\Theta_{ed}(u_1, \cdots, u_r) = {\Theta_{e} (u_1, \cdots, u_r)}^{\frac{p^{ed}-1}{p^e-1}}$ for every $d \geq 1$.
\item If $e = 1$ and $\alpha \leq 1$, then $\Theta_1$ has non-zero image in $\mathbb{F}_p[t_1, \cdots, t_r]$.
\end{enumerate}
\end{Proposition}

\begin{proof}
By definition, $f_{\Nsigma} = \sum_{i=1}^r u_i \cdot \x^{\aa_i}$.  By  Corollary \ref{Leftover: C}, $f^{(p^e-1) \cdot \alpha} \equiv \( f_{\Nsigma} \)^{(p^e-1) \cdot \alpha} \bmod \bracket{\m}{e}$, and is a multiple of $\x^{(p^e-1) \cdot \vone_m}$ modulo $\bracket{\m}{e}$.  The multinomial theorem then shows that $f^{\Nsigma}$ is $\Theta_e(u_1, \cdots, u_r) \cdot \x^{(p^e-1) \cdot \vone_m}$ mod $\bracket{\m}{e}$, where 
\begin{equation} \label{Thetae: e} \Theta_e(t_1, \cdots, t_r) := \sum_{\stackrel{{\bkappa \in \mathbb{N}^r}, {| \bkappa | = (p^e-1) \cdot \alpha}}{\EM \bkappa = (p^e-1) \cdot \vone_m}} \binom{(p^e-1) \cdot \alpha}{\kappa} t_1^{\kappa_1} \cdots t_r^{\kappa_r} \in \mathbb{Z}[t_1, \cdots, t_r]. \end{equation}

By assumption, there exists $\eeta = (\eta_1, \cdots, \eta_n) \in \Pmax$ with $(p^e-1) \cdot \eeta \in \mathbb{N}^n$.  By Lemma \ref{MaximalPointDP: L}, $\eta_{r+1} = \cdots = \eta_n = 0$ and $\EM \eeta = \vone_m$.  Thus, the index $\bkappa = (p^e-1) \cdot \eeta$ corresponds to a non-zero summand in \eqref{Thetae: e}, and so $\Theta_e \neq 0$.  

Next, set $\gamma_{d}:= \frac{p^{ed} - 1}{p^e-1}$. As $f^{(p^e-1) \cdot \alpha} = \Theta_e \(u_1, \cdots, u_r\) \cdot \x^{(p^e-1) \cdot \vone_m} + g$ for some $g \in \bracket{\m}{e}$,  \begin{equation} \label{coeffrel: e} f^{(p^{ed}-1) \alpha} = \( f^{(p^e-1) \cdot \alpha} \)^{\gamma_d} = \sum_{a+b = \gamma_d} \binom{\gamma_d}{a,b} \(\Theta_e(u_1, \cdots, u_r) \cdot \x^{(p^e-1) \cdot \vone_m} \)^a g^b.\end{equation}

Let $a,b$ be such that $a+b = \gamma_d$ and consider the base $p$ expansions $a = \sum_{w=1}^{(d-1)e} a_{w} \cdot p^{w}$ and $b = \sum_{w=1}^{(d-1)e} b_{w} \cdot p^{w}$.  As $\gamma_d = 1+ p^e + \cdots + p^{(d-2)e} + p^{(d-1)e}$ is the base $p$ expansion of $\gamma_d$, it follows from Lemma \ref{LucasLemma} that, modulo p, \begin{equation} \label{digitsums: e} \binom{\gamma_d}{a,b} \neq 0 \iff a_{e \ell} + b_{e \ell} = 1 \text{ for } 0 \leq \ell \leq d-1 \text{ and } a_w + b_w = 0 \text{ otherwise}. \end{equation}

Choose $a$ and $b$ with $b \neq 0$ such that $a+b = \gamma_d$ and $\binom{\gamma_d}{a,b} \neq 0 \bmod p$.  As $b \neq 0$, it follows from \eqref{digitsums: e} that $L = \max \set{ \ell : b_{e \ell} \neq 0}$ is well defined, and \eqref{digitsums: e} again shows that \[ a \geq p^{e(L+1)} + \cdots + p^{e(d-1)} \text{ and } b \geq p^{eL}.\]
As $g \in \bracket{\m}{e}$, it follows that $g^b \in \bracket{\m}{e(L+1)}$, and so 
\begin{align} \label{digitssum2: e}
\( \x^{(p^e-1) \cdot \vone_m}\)^a \cdot g^b = &  (x_1 \cdots x_m)^{(p^e-1) a } \cdot g^b \\ 
\in & \( x_1 \cdots x_m\)^{(p^e-1)( p^{e(L+1)} + \cdots + p^{e(d-1)})} \cdot \bracket{\m}{e(L+1)} \notag \\ 
= & \( x_1 \cdots x_m\)^{p^{ed}-p^{e(L+1)} }\cdot \bracket{\m}{e(L+1)} \subseteq \bracket{\m}{ed}. \notag
\end{align}

Thus, substituting the conclusion from \eqref{digitssum2: e} into \eqref{coeffrel: e}, we see that the only non-zero summand of $f^{(p^{ed}-1) \cdot \alpha}$ not contained in $\bracket{\m}{ed}$ corresponds to the indices $b = 0, a = \gamma_d$, i.e. $f^{(p^{ed}-1) \cdot \alpha} \equiv \( \Theta_e(u_1, \cdots, u_r) \cdot \x^{(p^e-1) \cdot \vone_m}\)^{\gamma_d}  \equiv \Theta_e(u_1, \cdots, u_r)^{\gamma_d} \cdot \x^{(p^{ed}-1) \cdot \vone_m} \bmod \bracket{\m}{ed}$.

For the last point, note that if $\alpha \leq 1$, then $(p-1) \alpha < p$, and hence the binomial coefficients in \eqref{Thetae: e} are non-zero modulo $p$.
\end{proof}

Let $f$ be as in Proposition \ref{SingleCoeff: P}, and suppose that $\Newton$ is in diagonal position, $(p-1) \cdot \eeta \in \mathbb{N}^n$ for some $\eeta \in \Pmax$, and $\alpha \leq 1$.   By  Proposition \ref{SingleCoeff: P}, the reduction of $\Theta_1$ modulo $p$ defines a non-empty closed set $Z \subseteq \mathbb{A}^n_{\L}$ satisfying the following condition:  If the coefficients of $f$ are not in $Z$, then $f^{(p-1) \cdot \alpha} \notin \bracket{\m}{}$.  By Lemma \ref{OneConditionSplittingLemma}, we conclude that $\fpt{f} \geq \alpha = \fpt{\sM}$, and it follows from Remark \ref{RationalityRemark} that $\fpt{f} = \fpt{\sM}$.  We now show that the condition that $\Newton$ be in diagonal position is not necessary to reach this conclusion.

\begin{Proposition}  
\label{GeneralFixedp: P}
Let $f$ be as in Proposition \ref{SingleCoeff: P}.  If $\fpt{\sM} \leq 1$ and $(p-1) \cdot \eeta \in \mathbb{N}^n$ for some $\eeta \in \Pmax$, there exists a non-empty closed set $Z \subseteq \mathbb{A}^n_K$ such that $\fpt{f} = \fpt{\sM}$ whenever $(u_1, \cdots, u_n)$, the coefficients of $f$, are not in $Z$.
\end{Proposition}

\begin{Remark}  The condition that $(p-1) \cdot \eeta \in \mathbb{N}^n$ in Proposition \ref{GeneralFixedp: P} is necessary.  Indeed, suppose that $\sM = \set{ x_1^{n}, \cdots, x_n^{n}}$.  We have seen in Example \ref{DiagonalPolytope: E} that $\Pmax = \set{ \( \frac{1}{n}, \cdots, \frac{1}{n} \)}$, so that $\fpt{\sM} = \frac{1}{n} + \cdots + \frac{1}{n} = 1$ by Proposition \ref{Monomialfpt: P}.  However, it is shown in \cite{Diagonals} that if $p \not \equiv 1 \bmod n$, then $\fpt{f} < \fpt{\sM}$ for every polynomial $f$ with $\Support(f) = \sM$.
\end{Remark}

\begin{proof}[Proof of Proposition \ref{GeneralFixedp: P}]
As explained in the  paragraph immediately preceding the statement of Proposition \ref{GeneralFixedp: P}, it suffices to define a non-empty closed set $Z \subseteq \mathbb{A}^n_K$ such that $f^{(p-1) \cdot \alpha} \notin \bracket{\m}{}$ whenever $(u_1, \cdots, u_n) \notin Z$.  Consider the polynomial 
\begin{equation} \label{GammaPoly: e} \Theta_{\eeta} (t_1, \cdots, t_n) = \sum_{\stackrel{ \k \in \mathbb{N}^n, | \k | = (p-1) \cdot \alpha}{\EM \k = (p-1) \EM \eeta}} \binom{ (p-1) \cdot \alpha }{k_1, \cdots, k_n} t_1^{k_1} \cdots t_n^{k_n} \in \mathbb{Z}[t_1, \cdots, t_n].\end{equation}
As $\alpha \leq 1$, each of the binomial coefficients in \eqref{GammaPoly: e} is not zero, and $\k = (p-1) \cdot \eeta$ corresponds to a summand of $\Theta_{\eeta}$, so that $\Theta_{\eeta} \neq 0 \bmod p$.  The multinomial theorem shows that $\Theta_{\eeta} \( u_1, \cdots, u_n \)$ is the coefficient of the monomial $\x^{ (p-1) \EM \eeta}$ (which is not in $\bracket{\m}{}$ as $\eeta \in \Pmax$) in $f^{(p-1) \alpha}$.  Thus, we may take $Z = \mathbb{V} \( \Theta_{\eeta} \) \subseteq \mathbb{A}^n_{\L}$.
\end{proof}

Let $f$ be a polynomial with $\Support(f) = \sM$.  By Remark \ref{RationalityRemark}, 
\begin{equation}
\label{maxboundexposition: e} \fpt{f} \leq \min \set{1, \fpt{\sM} }.
\end{equation}
We have seen in Proposition \ref{GeneralFixedp: P} that if $\fpt{\sM} \leq 1$ and $(p-1) \cdot \eeta \in \mathbb{N}^n$ for some $\eeta \in \Pmax$, then we have equality in \eqref{maxboundexposition: e} for a general choice of coefficients.  The following theorem shows that if $\fpt{\sM} > 1$ and $\fL$ has an isolated singularity at the origin, then equality in \eqref{maxboundexposition: e} holds provided that $p$ is large enough.

\begin{Lemma} \cite[Lemma 2.3]{Fed1983}
\label{Fedderfpt<1: L}
Let $f$ be a polynomial over a field of characteristic $p>0$ with $\Support(f) = \sM$.  Suppose $ \fpt{\sM} > 1 > \fpt{f}$  and that $f_{\Nsigma}$ has an isolated singularity at the origin, so that $\( x_1^N, \cdots, x_m^N\) \subseteq \( \frac{ \partial f_{\Nsigma} }{\partial x_1}, \cdots, \frac{ \partial f_{\Nsigma} }{\partial x_m}\)$ for some $N \geq 1$.  If $p^e > N$, then $\new{f_{\Nsigma}}{e} \geq \( p^e-N \) \cdot \fpt{\sM}$.  In particular, $p < \( N \cdot \frac{ \fpt{\sM} }{\fpt{\sM}-1} \)^{1/e}$.
\end{Lemma}

\begin{Remark}
Lemma \ref{Fedderfpt<1: L} is a generalization of \cite[Lemma 2.3]{Fed1983}, which applies to \emph{quasi-homogeneous} polynomials.  One may verify that $f$ is quasi-homogeneous in the sense of \cite{Fed1983} if and only if $\Newton$ is in diagonal position and $f = \fL$.
\end{Remark}


\begin{Remark}
\label{2dimS: R}
If $m = 2$, the hypothesis that $\fL$ have an isolated singularity at $\m$ is superfluous, as it is implied by the assumption that $\fpt{\sM} > 1$.  Indeed, we see from \eqref{monomialidealequalities: e} that $\fpt{\sM} > 1$ if and only if $(1,1) \in \Newton^{\circ}$, the interior of $\Newton$.  In this simplified setting, it is apparent that $(1,1) \in \Newton^{\circ}$ if and only if, after possibly reordering the variables, $\Nsigma= \set{ x_1, x_2^d}$ for some $d \geq 1$.  In this case, we see that $\Newton$ is in diagonal position, and that $f_{\LLambda}$ is a $K^{\ast}$-linear combination of $x_1$ and $x_2^d$, and is thus regular at all points.
\end{Remark}

\begin{proof}[Proof of Lemma \ref{Fedderfpt<1: L}]
As $\fpt{f} < 1$, it follows from setting $\lambda = 1$ in Lemma \ref{OneConditionSplittingLemma} that $f^{p^e-1} \in \bracket{\m}{e}$.  We claim that $\fL^{p^e-1} \in \bracket{\m}{e}$ as well.  By definition, there exists $g$ such that \begin{equation} \label{Fedder1: e} f^{p^e-1} = \fL^{p^e-1} + g. \end{equation}  If $\fL^{p^e-1} \notin \bracket{\m}{e}$, there exists a monomial $\Support(\fL^{p^e-1})$ not contained in $\bracket{\m}{e}$.  By definition such a monomial is of the form $\x^{\EM \bv}$ for some index $\bv = \sum_{i=1}^r v_i \cdot \be_i$.  As $f^{p^e-1} \in \bracket{\m}{e}$, it follows from \eqref{Fedder1: e} that there exists a monomial $\x^{\EM \bw} \in \Support(g)$ whose corresponding summand cancels the one determined by $\x^{\EM \bv}$, so that $|\bw| = |\bv|, \EM \bw = \EM \bv$.  It follows by definition of $g$ that $\bw = \sum_{i=1}^n w_i \cdot \be_i$, with $w_i \neq 0$ for some $i > r$.  As $L_{\LL} \equiv 1$ on $\Nsigma$,
\begin{align}
| \bv | = \sum_{i=1}^r v_i = \sum_{i=1}^r v_i \cdot L_{\LL} ( \aa_i ) = L_{\LL} \( \EM \bv \) & = L_{\LL} \( \EM \bw \)  = L_{\LL} \( \sum_{i=1}^r w_i \cdot \aa_i + \sum_{i=r+1}^n w_i \cdot \aa_i \) \notag \\  
& = \sum_{i=1}^r w_i + \sum_{i=r+1}^n w_i \cdot L(\aa_i) > | \bw | = | \bv |. \label{Fedder2: e}
\end{align}

\noindent In obtaining the contradiction in \eqref{Fedder2: e}, we used that $w_i \neq 0$ for some $i > r$ and that $L_{\LL} \( \aa_i \) > 1$ whenever $i > r$.  

Next, let $\new{\LL}{e} : = \new{\fL}{e}$ be as in \eqref{nu: deq}.  We have just shown that $\fL^{p^e-1} \in \bracket{\m}{e}$, and by definition $\new{\LL}{e} \leq p^e-2$.  As $\fL^{\new{\LL}{e} + 1} \in \bracket{\m}{e}$,  taking derivatives (keeping in mind that $\new{\LL}{e} \leq p^e-2$ and that we are over $\mathbb{F}_p$) shows that $\frac{\partial \fL}{\partial x_i} \cdot \fL^{\new{\LL}{e}} \in \bracket{\m}{e} \text{ for } 1 \leq i \leq n$.  It follows from this, and our original hypotheses, that 
\[ \fL^{\new{\LL}{e}} \in \( \bracket{\m}{e} : \frac{\partial \fL}{\partial x_1}, \cdots,  \frac{\partial \fL}{\partial x_m} \) \subseteq \( \bracket{\m}{e} : x_1^N, \cdots, x_m^N \) = \( x_1 \cdots x_m \)^{p^e-N}.\]
Thus, there exists $\k = \sum_{i=1}^r k_i \be_i \in \mathbb{N}^r$ with $| \k | = \new{\LL}{e}$ such that $\x^{\EM \k} \in \Support \( \fL^{\new{\LL}{e}} \)$ and $\EM \k \vgeq (p^e - N ) \cdot \vone_m$.  Applying $L_{\LL}$ to this inequality yields \[ \new{\LL}{e} =  \sum_{i=1}^n k_i \cdot L_{\LL} ( \aa_i) = L_{\LL} \( \EM \k \) \geq \( p^e - N \) \cdot L_{\LL} \( \vone_m \) = (p^e-N) \cdot \alpha, \] where we have again used that $L_{\LL} \( \vone_m \) = \alpha \cdot L_{\LL} \( \frac{1}{\alpha} \cdot \vone_m \) = \alpha$.  A straightforward manipulation shows that $\new{\LL}{e} \geq (p^e-N) \cdot \alpha$ if and only if $\alpha N \geq p^e \alpha - \new{\LL}{e}$, and so \begin{equation} \label{Fedder3: e} \alpha N \geq p^e \alpha - \new{\LL}{e}  = p^e (\alpha - 1) + p^e - \new{\LL}{e} > p^e (\alpha - 1), \end{equation}
where we have used that $\new{\LL}{e} \leq p^e-1$.  As we are assuming that $\alpha > 1$, \eqref{Fedder3: e} shows that $p^e <  \( \frac{\alpha}{\alpha-1} \cdot N \)$, and so we are done.
\end{proof}

\subsection{The unique maximal point condition}

\begin{Definition}
\label{UniqueMaximalPoint: D}
We say that $\P$ \emph{contains a unique maximal point} if $^{\#} \Pmax = 1$.
\end{Definition}

\begin{Remark}
\label{RationalCoordinates: R}
By definition, $\Pmax$ is a face of $\P$.  Thus, if $\Pmax = \set{ \eeta }$, then $\eeta$ must be a vertex of $\P$, and as such has rational coordinates.  Additionally, if $H = \set{ \s \in \mathbb{R}^n : L(\s) \leq 1}$ is a (lower) halfspace defined by a linear form $L$ with rational coefficients, then every vertex of the rational polyhedral set $\P \cap H$ must also have rational coordinates.
\end{Remark}

\begin{Example}
The polytope from Example \ref{DiagonalPolytope: E} always has a unique maximal point.  If $\P$ is the polytope from Example \ref{PyramidPolytope: E}, we see that $\P$ contains a unique maximal point if and only if $| \vee_3 | = \frac{1}{a} + \frac{1}{b} \neq \frac{1}{c} = | \vee_4 |$, in which case $\Pmax = \set{\vee_3}$ or $\set{\vee_4}$.  If $\frac{1}{a} + \frac{1}{b} = \frac{1}{c}$, then $\Pmax$ is the edge connecting $\vee_3$ and $\vee_4$.
\end{Example}


\begin{Remark}  Recall that a set $\set{ \bb_1, \cdots, \bb_r}$ is said to be \emph{affinely independent} if $\0$ is the unique solution to the system of equations $\sum_{i=1}^n k_i \cdot \bb_i = \0$ and $\sum_{i=1}^n k_i = 0$.  If $\Newton$ is in diagonal position, it follows from Proposition \ref{MaximalPointDP: L} that $\P$ contains a unique maximal point if and only if there is a unique solution to the system $\s \vgeq 0, \EM \s = \vone_m$, and $s_{r+1} = \cdots = s_n = 0$.  In this case, we see that $^{\#}\Pmax = 1$ if the exponents of $\sM_{\Nsigma}$ are affinely independent.  
\end{Remark}

\begin{Lemma}  
\label{UniqueCoeffLemma}
Suppose  $\P$ has a unique maximal point $\eeta \in \P$.
\begin{enumerate}
\item If  $| \s | = | \tr{\eeta}{e} |$ and $ \EM \s = \EM \tr{\eeta}{e}$ for some $\s \vgeq \0$, then $\s = \tr{\eeta}{e}$.
\item If $|\s | =| \n |, \EM \s = \EM \n$, and $\n \vleq \tr{\eeta}{e}$ for some $\n, \s \vgeq \0$, then $\s = \n$.
\end{enumerate}
\end{Lemma}


\begin{proof}
To prove the first statement, let $\eeta' := \s + \eeta - \tr{\eeta}{e}$.  By hypothesis, $\eeta' \vgeq \s \vgeq \0$,   $\EM \eeta' = \EM \s + \EM \eeta - \EM \tr{\eeta}{e} = \EM \eeta$, and $| \eeta' | = |\s| + | \eeta | - | \tr{\eeta}{e}| =  | \eeta |$, which shows that  $\eeta'$ is a maximal point of $\P$.  Thus $\eeta'=\eeta$, and $\s = \tr{\eeta}{e}$.  

For the second statement, let $\s' := \s + \tr{\eeta}{e} - \n$. By hypothesis, $\s'  \vgeq \0, |\s'| = | \tr{\eeta}{e} |$, and $\EM \s' = \EM \tr{\eeta}{e}$. The first statement, applied  to $\s'$, shows that $\s' = \tr{\eeta}{e}$, and thus $\s = \n$.
\end{proof}

\begin{Corollary}
\label{UniqueCoeffCorollary}  Suppose $\P$ has a unique maximal point $\eeta \in \P$, and let $f$ be a polynomial with $\Support(f) = \sM$ and coefficients $u_1, \cdots, u_n$.  
\begin{enumerate}  
\item The coefficient of the monomial $\x^{p^e \EM \tr{\eeta}{e}}$ in  $f^{ p^e | \tr{\eeta}{e} | }$
is $\binom{p^e | \tr{\eeta}{e} | }{p^e \tr{\eeta}{e}} \uu^{p^e  \tr{\eeta}{e}}.$
\item If $\n \in \frac{1}{p^e} \cdot \mathbb{N}^n$ is an index such that $\n \vleq \tr{\eeta}{e}$, then the coefficient of the monomial $\x^{p^e \EM \n}$ in $f^{ p^e |\n| }$ is $\binom{ p^e | \n | }{p^e \n } \uu^{p^e \n}$.
\end{enumerate}
\end{Corollary}

\begin{Corollary}  
\label{Converse: C}
Suppose that $\Newton$ is in diagonal position, $\P$ contains a unique maximal point $\eeta$, and $p$ does not divide any of the denominators in $\eeta$.  If $f$ is a polynomial with $\Support(f) = \sM$ and coefficients $u_1, \cdots, u_n$, then there exist infinitely many $e$ such that $(p^e-1) \cdot \eeta = p^e \tr{\eeta}{e}$ and $f^{(p^e-1) \cdot \alpha} = f^{p^e |\tr{\eeta}{e}| }\equiv \binom{ p^e | \tr{\eeta}{e} | }{ \tr{\eeta}{e} } \uu^{p^e \tr{\eeta}{e}} \cdot \x^{(p^e-1) \cdot \vone_m} \bmod \bracket{\m}{e}$.
\end{Corollary}

\begin{proof}
If $p$ does not divide $d$, then $p$ has finite order in the multiplicative group $\( \mathbb{Z} / d \mathbb{Z} \)^{\times}$.  It follows that there exists an $e$ (and hence, infinitely many $e$) with $(p^e-1) \cdot \eeta \in \mathbb{N}^n$, and Lemma \ref{SimplifiedTruncation: L} implies that $p^e \tr{\eeta}{e} = (p^e-1) \cdot \eeta$ for such  $e$.  Applying Proposition \ref{SingleCoeff: P} shows that $f^{p^e |\tr{\eeta}{e}|}$ is congruent to $ \gamma \cdot \x^{(p^e-1) \cdot \vone_m}$ modulo $\bracket{\m}{e}$; that $\gamma = \binom{ p^e | \tr{\eeta}{e} |}{p^e \tr{\eeta}{e}} \uu^{p^e \tr{\eeta}{e}}$ may then be deduced from \eqref{Thetae: e} or Corollary \ref{UniqueCoeffCorollary}.
\end{proof}

\section{Proof of the Main Theorem}
\label{Proof: S}

\begin{Notation}  $\sM, \P$, and $\Newton$ will continue to be as in the previous section.  Furthermore, $\L$ will denote an $F$-finite field of characteristic $p$, and $f$ will denote a polynomial over $\L$ with $\Support(f) = \sM$.  
\end{Notation}

\begin{Theorem}
\label{MainTheorem} Suppose that $\P$ contains a unique maximal point $\eeta = (\eta_1, \cdots, \eta_n) \in \P$, and let $L = \sup \{ e:  \digit{\eta_1}{d} + \cdots + \digit{\eta_n}{d} \leq p-1 \text{ for } 0 \leq d \leq e\}$,  where $\digit{\eta_i}{d}$ is the $d^{\th}$ digit of $\eta_i$.
\begin{enumerate} 
\item \label{MainTheoremI} If $L = \infty$, then $\lfpt{\m}{f} = \lfpt{\m}{\sM}$.  The converse holds if $\Newton$ is in diagonal position and $p$ does not divide any of the denominators of $\eeta$.
\item \label{MainTheoremII} If $L < \infty$, then $\lfpt{\m}{f} \geq \tr{\eta_1}{L}+ \cdots +\tr{\eta_n}{L}+ \frac{1}{p^L}$.
\end{enumerate}
\end{Theorem}

\begin{proof}[Proof of Theorem \ref{MainTheorem}]  
 Write $f = \sum_{i=1}^n u_i \x^{\aa_i}$ as a $\L^{\ast}$-linear combination of the monomials in $\sM$.   We first prove $\eqref{MainTheoremI}$, and thus assume that entries of $\eeta$ add without carrying (in base $p$).  Corollary \ref{UniqueCoeffCorollary} give that, after gathering terms, \begin{equation} \label{MTProof1: e} \binom{ p^e | \tr{\eeta}{e} | }{p^e \tr{\eeta}{e}} \uu^{p^e \tr{\eeta}{e}} \x^{p^e\EM \tr{\eeta}{e}} \end{equation} appears as a summand of $f^{ p^e | \tr{\eeta}{e} | }$.   By definition, each $u_i \in \L^{\ast}$, so $\uu^{p^e \tr{\t}{e}} \neq 0$, while the assumption on the entries of $\eeta$ implies the integers $p^e \tr{\eta_1}{e}, \cdots, p^e \tr{\eta_n}{e}$ add without carrying as well, and applying Lemma \ref{LucasLemma} then shows that $\binom{ p^e | \tr{\eeta}{e} | }{p^e \tr{\eeta}{e}} \neq 0 \bmod p$.  Finally, $\tr{\eeta}{e} \vleq \eeta$, and so every entry of $p^e \EM \tr{\eeta}{e}$ is less than $p^e-1$.  We see then that the monomial in \eqref{MTProof1: e} is in $\Support \( f^{p^e | \tr{\eeta}{e} |} \)$, but not in  $\m^{[p^e]}$.  Thus, $f^{p^e | \tr{\eeta}{e} | } \notin \m^{[p^e]}$, and so $\frac{\new{f}{e}}{p^e} \geq | \tr{\eeta}{e}|$.  Letting $e \to \infty$ shows that $\fpt{f} \geq | \eeta |$, while the opposite inequality holds by Remark \ref{RationalityRemark}. 

We now address the second statement of the first point.  Suppose then, by means of contradiction, that $L < \infty$, yet $\fpt{f} = \fpt{\sM}  = \alpha$.  By Corollary \ref{Converse: C}, there exists $e > L$ such that $(p^e-1) \cdot \eeta \in \mathbb{N}$ and  \begin{equation} \label{conversefinally1: e} f^{(p^e-1) \cdot \alpha } = f^{p^e \tr{\eeta}{e} } \equiv \binom{ p^e | \tr{\eeta}{e} | }{p^e \tr{\eeta}{e}} \uu^{p^e \tr{\eeta}{e}} \cdot \x^{(p^e-1) \cdot \vone_m} \bmod \bracket{\m}{e}.\end{equation}

As $e > L$, it follows from the definition of $L$ that the entries of $p^e \tr{\eeta}{e}$ \emph{do not} add without carrying, so that $\binom{ p^e | \tr{\eeta}{e} | }{p^e \tr{\eeta}{e}} \equiv 0 \bmod p$ by Lemma \ref{LucasLemma}. It then follows from \eqref{conversefinally1: e} that \begin{equation} \label{conversecontradiction: e} f^{(p^e-1) \cdot \alpha} \in \bracket{\m}{e}. \end{equation}  However, the assumption that $\fpt{f} = \alpha$ implies that $\alpha \in (0,1]$, and as $(p^e-1) \cdot \alpha \in \mathbb{N}$,  Lemma \ref{OneConditionSplittingLemma} implies that $f^{(p^e-1) \cdot \alpha} \notin \bracket{\m}{e}$, contradicting \eqref{conversefinally1: e}.  

We now address the second point.  Recall that, by convention, $\digit{\eta_i}{0} = 0$, so that $L \geq 0$.  Furthermore, it follows from the definition of $L$ that $\digit{\eta_1}{L+1} + \cdots + \digit{\eta_n}{L+1} \geq p$.  After subtracting off from each $\digit{\eta_i}{L+1}$ as necessary, we obtain integers $\delta_1, \cdots, \delta_n$ such that \begin{equation} \label{deltaproof: e}\delta_1 + \cdots + \delta_n = p-1 \text{ and }0 \leq \delta_i \leq \digit{\eta_i}{L+1},\end{equation}  with the second inequality in \eqref{deltaproof: e}  being strict for at least one index.  Without loss of generality, we assume that $\delta_1 < \digit{\eta_1}{L+1}$. For $e \geq L+2$, set 
\begin{equation} \label{lambdadef} \ll(e) = \tr{\eeta}{L} +\left( \frac{\delta_1}{p^{L+1}} + \frac{p-1}{p^{L+2}} + \cdots + \frac{p-1}{p^e}, \frac{\delta_2}{p^{L+1}},  \cdots, \frac{\delta_n}{p^{L+1}} \right) . \end{equation}

We now summarize some important properties of $\ll(e)$:  By construction, $\ll(e) \in \frac{1}{p^e} \cdot \mathbb{N}^n$, and the definition of the $\delta_i$ (along with our assumption that $\delta_i < \digit{\eta_1}{L+1}$) shows that $\ll(e) \vl \tr{\eeta}{e}$.  Additionally, as $\delta_1 + \cdots + \delta_n = p-1$, it follows from the definition of $L$ that the entries of $p^e \cdot \ll(e)$ add without carrying (in base $p$). Finally, we have  
\[ | \ll(e) |= | \tr{\eeta}{L}| + \frac{1}{p^{L+1}} \cdot \left( \sum_{i=1}^n\delta_i \right) +  \frac{p-1}{p^{L+2}} + \cdots + \frac{p-1}{p^e}  =  | \tr{\eeta}{L}| + \frac{p-1}{p^{L+1}} + \frac{p-1}{p^{L+2}} + \cdots + \frac{p-1}{p^e}.   \]

These properties imply that $\binom{ p^e | \ll(e) | }{ p^e \ll(e) } \uu^{p^e \ll(e)} \x^{ p^e \EM \ll(e)}$ is a non-zero summand of $f^{p^e | \ll(e) |}$ and is not contained in $\bracket{\m}{e}$, so $f^{p^e | \ll(e) |} \notin \m^{[p^e]}$.  Thus, $\frac{\new{f}{e}}{p^e} \geq | \ll(e) | = | \tr{\eeta}{L} | + \sum_{d=L+1}^e \frac{p-1}{p^{d}}$, and the assertion follows by letting $e \to \infty$.
\end{proof}

\begin{Remark}  The estimates given in Theorem \ref{MainTheorem} can be used to calculate $\lfpt{\m}{f}$ \emph{in any characteristic} whenever $f$ is either a diagonal or binomial hypersurface \cite{Diagonals,Binomials}.
\end{Remark}

\section{Log canonical singularities and dense $F$-pure type}
\label{LCDFPT: S}

Throughout this section, $S$ will denote the polynomial ring $\mathbb{C}[x_1, \cdots, x_m]$, and $\sM$ will continue to denote $\set{ \x^{\aa_1}, \cdots, \x^{\aa_n}}$, a collection of $n$ monomials in the variables $x_1, \cdots, x_m$.  

\subsection{The log canonical threshold of a polynomial}

\begin{Definition}
\label{lct: D}  Let $f \in S$ be a non-zero polynomial such that $f(\0) = 0$.  Then,  \[ \lct{f}: = \sup \set{ \lambda > 0 : \frac{1}{|f|^{2 \lambda}} \text{ is locally integrable at } \0 }\] exists, and is called the \emph{log canonical threshold of $f$ at $\0$}.
\end{Definition}

The invariant $\llct{f}$ can be thought of as measuring the singularities of $f$ near $\0$, with smaller values corresponding to ``worse'' singularities.  One important property of these invariants is that $\lct{f} \in (0,1] \cap \mathbb{Q}$.  Though this is not obvious from Definition \ref{lct: D}, it follows immediately from an alternate characterization of $\lct{f}$  in terms of (log) resolution of singularities.  For this alternate definition, see the survey \cite{EM2006}.  

By imposing similar local integrability conditions on the members of an ideal $\a$ of $S$, one may define $\lct{\a}$, the log canonical threshold at $\0$ of $\a$, as follows:  If $\a = (f_1, \cdots, f_d)$, then  \begin{equation} \label{lctmon: e} \lct{\a} = \sup \set{ \lambda > 0 : \( |f_1|^2 + \cdots + |f_d|^2\)^{-\lambda} \text{ is locally integrable at } \0}. \end{equation}

Using resolution of singularities, one sees that $\lct{\a}$ is well defined and rational, and apparently, $\lct{f} \leq \lct{\a}$ if $f \in \a$.  We gather these facts, along with a familiar formula for $\fpt{\sM}$ which may be deduced from \cite[Example 5]{Howald2001a} below.

\begin{Proposition} \cite{Howald2001a}
\label{lctBP: P}
Let $f$ be a polynomial over $\mathbb{C}$ with $\Support(f) = \sM$.  Then, $\lct{f}$ and $\lct{\a}$ are rational numbers, and $\lct{f} \leq \min \set{ 1, \lct{\sM}}$.  We also have that $\lct{\sM}=  \max \set{ \lambda > 0 : \frac{1}{\lambda} \cdot \vone_m \in \Newton }$. 
\end{Proposition}

Comparing the statements in Proposition \ref{lctBP: P} with those in Remark \ref{RationalityRemark} shows that log canonical thresholds and $F$-pure threshold satisfy seemingly dual conditions.  Furthermore, it follows from \eqref{monomialidealequalities: e} and Proposition \ref{lctBP: P} that 
\begin{equation} \label{allequal: e} \fpt{\sM} = \max \set{ | \s | : \s \in \P } = \max \set{ \lambda > 0 : \frac{1}{\lambda} \cdot \vone_m \in \Newton } = \lct{\sM}, \end{equation} an observation (well-known among experts) that reveals the first of many deep connections between $F$-pure and log canonical thresholds.  In order to precisely state this relationship, we review the process of \emph{reduction to positive characteristic.}

\subsection{On reduction to positive characteristic}

If $f \in \mathbb{Q}[x_1, \cdots, x_m]$, we may reduce the coefficients of $f$ modulo $p \gg 0$ to obtain a \emph{family of positive characteristic models} $\set{ f_p }$ of $f$ over the finite fields  $\mathbb{F}_p$ for all but finitely many $p$.  Instead suppose that $f \in \mathbb{C}[x_1, \cdots, x_m]$ but does not have rational coefficients.  Let $A$ be a finitely-generated  $\mathbb{Z}$-algebra containing the coefficients of $f$, so that $f \in A[x_1, \cdots, x_m]$.  For such an algebra, $A / \mu$ is a finite field for every $\mu \in \mSpec A$, and all but finitely many primes appear in the set $\set{ \operatorname{char} A / \mu : \mu \in \mSpec A}$.  Let $f_{\mu}$ denote the image of $f$ in $\( A / \mu \)[x_1, \cdots, x_m]$.  If $f(\0) = 0$, we may enlarge $A$ (say, by adjoining the inverses of the coefficients of $f$) so as to assume that $\Support(f_{\mu}) = \Support(f)$ for every $\mu \in \mSpec A$.  We again call the set $\set{f_{\mu}: \mu \in \mSpec A}$ a family of positive characteristic models of $f$.  In Corollary \ref{Characteristic: C}, we justify some these assertions, and our main tool will be the following variant of Noether Normalization.

\begin{Lemma}
\label{Noether}
Let $A$ be a finitely-generated algebra over a domain $D$.  There exists non-zero element $N \in D$ such that $D_N \subseteq A_N$ factors as $D_N \subseteq D_N[z_1, \cdots, z_d] \subseteq A_N$, where  $z_1, \cdots, z_d$ are algebraically independent over $D_N$, and $D_N[z_1, \cdots, z_d] \subseteq A_N$ is finite.
\end{Lemma}

If $L = \operatorname{Frac} D$ and $R$ is the localization of $A$ at the non-zero elements of $D$, then Lemma \ref{Noether} can be obtained by applying the Noether Normalization theorem to the inclusion $L \subseteq R$.  See \cite{HochsterNoether} for an alternate proof that does not rely on Normalization for algebras over a field.

\begin{Corollary}  
\label{Characteristic: C}
Every maximal ideal of a finitely-gnerated $\mathbb{Z}$-algebra $A$ contains a prime $p$, and $A / \mu$ is a finite field for every maximal ideal $\mu \subseteq A$.  Furthermore, all but finitely many primes $p$ are contained in a maximal ideal of $A$.
\end{Corollary}

\begin{proof}
Let $\mu \subseteq A$ be a maximal ideal.  If $\mu \cap \Z = 0$, then $A / \mu$ is also a finitely generated $\Z$-algebra.  By Lemma \ref{Noether}, $A / \mu = \( A / \mu \)_N$ is module finite over a polynomial ring with coefficients in $\Z_N$ for some $N$, so that $0 = \dim A / \mu \geq \dim \Z_N = 1$, a contradiction. If $p \in \mu$, then $A/\mu$ is finitely-generated over $\mathbb{F}_p$, and thus is module finite over a polynomial ring $\mathbb{F}_p[z_1, \cdots, x_d]$ by Lemma \ref{Noether}.  As $A / \mu$ is a field, dimension considerations force that $d = 0$, and thus $A / \mu$ is finite over $\mathbb{F}_p$.  Finally, consider a factorization $\Z_N \subseteq \Z_N[z_1, \cdots, z_d] \subseteq A_N$ as in Lemma \ref{Noether}.  Every $p$ not dividing $N$ generates a prime ideal in the polynomial ring $\Z_N[z_1, \cdots, z_d]$, and so by the Lying Over Theorem, there exists a prime (and hence maximal) ideal of $A$ not containing $N$ and lying over $p$. 
\end{proof}

\begin{Corollary}
\label{DenseCorollary}
Let $A$ be a finitely-generated algebra over a domain $D$.  Then, the inverse image of a dense set under the induced map $\Spec A \stackrel{\pi}{\to} \Spec D$ is also dense.
\end{Corollary}

\begin{proof}  
Let $\GG$ be dense in $\Spec D$.  It suffices to show that $\Spec A_f \cap \pi^{-1}( \GG)$ is non-empty for every non-zero $f \in A$.  As $A$ is finitely generated over $D$, so is $A_f = A[T] / (1-Tf)$.  Consider a factorization $D_N \subseteq D_N[z_1, \cdots, z_d] \subseteq A_{fN}$ as in Lemma \ref{Noether}.  By the Lying Over Theorem, $\Spec A_{fN} \stackrel{\pi}{\to} \Spec D_N$ is surjective.  As $\GG$ is dense, $\GG \cap \Spec D_N = \GG \cap \pi \( \Spec A_{fN} \)$ is non-empty. Consequently, $\Spec A_{fN} \cap \pi^{-1}(\GG)$, and hence $\Spec A_f \cap \pi^{-1} (\GG)$, is non-empty.
\end{proof}

\subsection{Connections with $F$-pure thresholds}

\begin{Notation}  
Let $A$ be a finitely generated $\Z$ sub-algebra of $\mathbb{C}$.  We use $S_A$ to denote the subring $A[x_1, \cdots, x_m] \subseteq S$; note that $\mathbb{C} \tens_{A} S_A = S$.  If $\mu$ is a maximal ideal of $A$, $S_A(\mu)$ denotes the polynomial ring $S_A \tens_A A / \mu = S_A / \mu S_A = \left( A / \mu \right) [x_1, \cdots, x_m]$.  By Corollary \ref{Characteristic: C},  $\characteristic S_A(\mu) > 0$.  For $g \in S_A$,  $g_{\mu}$ denotes the image of $g$ in $S_A(\mu)$.  Finally, $\m$ will denote the ideal generated by the variables $x_1, \cdots, x_m$ in the polynomial rings  $S, S_A$, and $S_A ( \mu) $.  
\end{Notation}

It is an important fact that the $F$-pure  (respectively, log canonical) threshold of a polynomial may also be defined in terms of its associated \emph{test ideals} (respectively, \emph{multiplier ideals}).  Theorem \ref{LimitsofFPT: T} below was first observed in \cite[Theorem 3.4]{MTW2005}, and follows from deep theorems in \cite{Karen2000, HY2003} relating test ideals and multiplier ideals.  We refer the reader to the author's thesis for a  detailed discussion of how to deduce Theorem \ref{LimitsofFPT: T} from the results of \cite{Karen2000, HY2003}.

\begin{Theorem}  
\label{LimitsofFPT: T}
Let $f \in S$ be a polynomial with $f(\0) = 0$.  Then, for every finitely generated $\mathbb{Z}$-algebra $A \subseteq \mathbb{C}$ with $f \in S_A$, the following hold:
\begin{enumerate}
\item There exists a dense open set $U \subseteq \Spec A$ such that $\fpt{f_{\mu}} \leq \lct{f}$ for every maximal ideal $\mu \in U$.
\item For every $0 < \lambda < \llct{f}$, there exists a dense open set $U_{\lambda} \subseteq \Spec A$ such that $\lambda \leq \fpt{f_{\mu}} \leq \lct{f}$ for every maximal ideal $\mu \in U_{\lambda}$.
\end{enumerate}
\end{Theorem}

We stress that the open set $U_{\lambda}$ depends on $\lambda$, and often shrinks as $\lambda$ increases.

\begin{Remark}
\label{SimpleLimits: R}
Suppose that $f$ has integer coefficients.  If $f_p$ denotes the image of $f$ in $\mathbb{F}_p[x_1, \cdots, x_m]$, the statements of Theorem \ref{LimitsofFPT: T} become 
\[ \fpt{f_p} \leq \lct{f} \text{ for } p \gg 0 \text{ and } \lim_{p \to \infty} \fpt{f_p} \leq \lct{f}.\]
\end{Remark}

The behavior illustrated in Remark \ref{SimpleLimits: R} is apparent in Example \ref{CuspExample: E}, which also shows  that $\fpt{\(x^2 + y^3\)_p} = \lct{x^2 + y^3}$ for a dense, though not \emph{open}, set of primes in $\Spec \mathbb{Z}$.  This  kind of behavior is of particular interest, and motivates the following definition.

\begin{Definition}
\label{LCDFPT: D}
Let $f \in S$ with $f(\0)=0$.  We say \emph{log canonicity equals dense $F$-purity} for $f$ if for every finitely-generated $\mathbb{Z}$-algebra $A \subseteq \mathbb{C}$ with $f \in S_A$, there exists a dense subset $W \subseteq \Spec A$ such that $\fpt{f_{\mu}} = \lct{f}$ for every maximal ideal $\mu \in W$.   If $W$ is open in $\Spec A$, we say that log canonicity equals \emph{open} $F$-purity for $f$.
\end{Definition}

\begin{Remark}
\label{OneCoefficientAlgebra: R}
To show that log canonicity equals (open/dense) $F$-purity for $f$, it suffices to produce a \emph{single} finitely-generated $\mathbb{Z}$-algebra $A$ satisfying the conditions of Definition \ref{LCDFPT: D}.  We refer the reader to the author's thesis for a detailed verification of this.
\end{Remark}

\begin{Remark}
\label{TerminologyJustification: R}
  In the study of  \emph{singularities of pairs}, the terms ``log canonical'' and ``$F$-purity'' have their own independent meanings.  Indeed, one defines the notion of log singularities for \emph{pairs} $\pair{S}{f}{\lambda}$ via resolution of singularities (or via integrability conditions related to those in Definition \ref{lct: D}) \cite{Laz2004}.  Additionally, we have that \[ \lct{f} = \sup \set{ \lambda > 0 : \pair{S}{f}{\lambda} \text{ is log canonical at $\0$}},\] which justifies the use of the term ``log canonical threshold.''    In the positive characteristic setting, one defines the notion of $F$-purity for pairs via the Frobenius morphism, and we again have that the $F$-pure threshold of a polynomial is the supremum over all parameters such that the corresponding pair is $F$-pure \cite{HW2002, TW2004}.  We say that the pair $\pair{S}{f}{\lambda}$ is of \emph{dense $F$-pure type} if for every (equivalently, for some) finitely generated $\mathbb{Z}$-algebra $A \subseteq \mathbb{C}$ with $f \in S_A$, there exists a dense set $W \subseteq \Spec A$ such that the pair $\pair{S_A(\mu)}{f_{\mu}}{\lambda}$ is $F$-pure for every maximal ideal $\mu \in W$.  It is shown in \cite{HW2002, Tak2004} that if $\pair{S}{f}{\lambda}$ is log canonical, then it is also of dense $F$-pure type.

 It is an important, yet easy to verify, property of log canonicity that $\pair{S}{f}{\lct{f}}$ is log canonical at $\0$.  Consequently, $\pair{S}{f}{\lambda}$ is log canonical if and only if $0 \leq \lambda \leq \lct{f}$.  In prime characteristic, that a pair is  $F$-pure at the threshold is shown in  \cite{Hara2006, Singularities}, and it follows that the reductions $\pair{S_A(\mu)}{f_{\mu}}{\lambda}$ are $F$-pure if and only if $0 \leq \lambda \leq \fpt{f_{\mu}}$.  Examining the definitions, we reach the following conclusion:  To show that log canonicity is equivalent to dense $F$-pure type for pairs $\pair{S}{f}{\lambda}$, it suffices to show that $\fpt{f_{\mu}} = \lct{f}$ for all maximal $\mu$ in some dense subset of $\Spec A$, which justifies our choice of terminology in Definition \ref{LCDFPT: D}.
\end{Remark}

\begin{Proposition} \cite[Theorem 2.5]{Fed1983}
\label{bigmonomialLCDFPT: P}
Let $f \in S$ be a polynomial with $\Support(f) = \sM$,and let $\alpha$ continue to denote the common value $\fpt{\sM} = \lct{\sM}$.  Let $A \subseteq \mathbb{C}$ be a finitely-generated $\mathbb{Z}$-algebra such that $f \in S_A$.  If $\alpha > 1$ and $\fL$ has an isolated singularity at $\0$, there exists a dense open subset $U \subseteq \Spec A$ such that $\fpt{f_{\mu}} = 1$ for every maximal ideal $\mu \in U$.  In particular, $\lct{f} = 1$, and log canonicity  equals open $F$-purity for $f$.
\end{Proposition}

\begin{Remark}
Proposition \ref{bigmonomialLCDFPT: P} generalizes \cite[Theorem 2.5]{Fed1983}.  By Remark \ref{2dimS: R}, the assumption that $\fL$ have an isolated singularity is unnecessary in dimension two.
\end{Remark}

\begin{proof}[Proof of Proposition \ref{bigmonomialLCDFPT: P}]

As discussed in Remark \ref{2dimS: R}, it suffices to produce a single algebra $A$ satisfying the conditions of the proposition.  By hypothesis,  \begin{equation} \label{isosingconsequence: e} (x_1^N, \cdots, x_m^N) \subseteq \( \frac{\partial \fL}{\partial x_1}, \cdots, \frac{ \partial\fL}{\partial x_m} \) \end{equation} for some $N \geq 1$.  Let $A$ be the finitely-generated $\mathbb{Z}$-subalgebra of $\mathbb{C}$ obtained by adjoining to $\mathbb{Z}$ the coefficients of $f$ and their inverses (so that $\Support{f_{\mu}} = \sM$ for every $\mu \in \Spec A$), and well as all of the coefficients needed to express every $x_i^N$ as a linear combination of the partial derivatives of $\fL$.  By construction, $f \in S_A$, and \eqref{isosingconsequence: e} also holds in $S_A$, and hence in $S_A(\mu)$ for every maximal ideal $\mu \in \Spec A$.  

Let $\mu \in \Spec A$ be a maximal ideal.  If $\fpt{f_{\mu}} < 1$, it follows from Lemma \ref{Fedderfpt<1: L} that $\operatorname{char} A / \mu < N \cdot \frac{\alpha}{\alpha - 1}$.  By Corollary \ref{Characteristic: C},  $\set{ \mu : \operatorname{char} A / \mu > N \cdot \frac{\alpha}{\alpha -1 }}$ is the set of all maximal ideals in some non-empty open set $U_{\circ}$ of $\Spec A$, and the first claim follows.  Next, choose a dense open set $U \subseteq \Spec A$ satisfying the conditions of Theorem \ref{LimitsofFPT: T}.  For every $\mu \in U_{\circ} \cap U$, $1 = \fpt{f_{\mu}} \leq \lct{f} \leq \min \set{1, \alpha} = 1$, and so we are done.
\end{proof}

\begin{Lemma}
\label{AlternateLemma}
Let $f \in S$ with $\Support(f) = \sM$.  Then, log canonicity equals dense $F$-purity for $f $ if there exists a finitely-generated $\Z$-algebra $A \subseteq \mathbb{C}$ with $f \in S_A$, an infinite set of primes $\GG$, and for every $p \in \GG$ a set $W_p$ satisfying the following conditions:  
\begin{enumerate}
\item $W_p$ is a dense in $\pi^{-1}(p)$, where $\Spec A \stackrel{\pi}{\longrightarrow} \Spec \Z$ denotes the map induced by  $\Z \subseteq A$.
\item $\fpt{f_{\mu}} = \min \set{1, \fpt{\sM}}$ for every maximal ideal $\mu \in W_p$.
\end{enumerate}
\end{Lemma}

\begin{proof}   
Let $W = \bigcup_{p \in \GG} W_p$.  As $W_p$ is dense in $\pi^{-1}(p)$,  $\overline{W_p} = \pi^{-1}(p)$, and thus \begin{equation} \label{closure} \overline{W} \supseteq \bigcup_{p \in \GG} \overline{W_p} = \bigcup_{p \in \GG} \pi^{-1}(p) = \pi^{-1}(\GG). \end{equation} 

By Corollary \ref{DenseCorollary}, $\pi^{-1}(\GG)$ is dense in $\Spec A$, and applying \eqref{closure} shows that $W$ is dense as well.  Let $U \subseteq \Spec A$ be the dense open set given by Theorem \ref{LimitsofFPT: T}.  As $W$ is dense and $U$ is dense and open, it follows that $U \cap W$ is dense in $\Spec A$.  Furthermore, for $\mu \in U \cap W$, 
\begin{equation} \label{otherinequality} \lct{f} \leq \min \set{1,\lct{\sM}} = \min \set{1,\fpt{\sM}} = \fpt{f_{\mu}} \leq \lct{f}.\end{equation}  
Indeed, the leftmost inequality in \eqref{otherinequality} holds by Proposition \ref{lctBP: P}, the first equality by \eqref{allequal: e}, the second equality by our assumption on $W$, and the rightmost inequality by the defining property of $U$.  We conclude from \eqref{otherinequality} that $\fpt{f_{\mu}} = \lct{f}$ for every maximal ideal $\mu$ in the dense subset  $U \cap W$ of $\Spec A$, and the claim follows.
\end{proof}

\begin{Theorem}
\label{UniquePointTheorem}
 If $\P$ contains a unique maximal point $\eeta$, then log canonicity equals dense $F$-purity for every polynomial $f \in S$ with $\Support(f) = \sM$.
\end{Theorem}

\begin{proof}  Let $A$ be such that $f \in S_A$ and consider the map $\Spec A \stackrel{\pi}{\to} \Spec \Z$. After enlarging $A$, we may assume that $\Support(f_{\mu}) = \sM$ for every $\mu \in \mSpec A$.  If $|\eeta| \leq 1$, let $\GG$ consist of all primes $p$ such $\pi^{-1}(p)$ is non-empty, and such that the entries of $\eeta$ add without carrying (in base $p$).  By Corollary \ref{Characteristic: C} and Lemma \ref{ConstantExpansionLemma}, $^{\#} \GG = \infty$, and it follows from Theorem \ref{MainTheorem} that $\fpt{f_{\mu_p}} = |\eeta| = \fpt{\sM}$ for all $p \in \GG$ and $\mu_p \in \pi^{-1}(p)$. If $| \eeta | > 1$, instead let $\GG$ denote the set of all primes $p$ such that $\pi^{-1}(p)$ is non-empty and \begin{equation} \label{condition} \digit{\eta_1}{e} + \cdots + \digit{\eta_n}{e} \geq p \text{ for every } e \geq 1. \end{equation} Again, Corollary \ref{Characteristic: C} and Lemma \ref{ConstantExpansionLemma} imply $^{\#} \GG = \infty$, and \eqref{condition} allows us to apply Theorem \ref{MainTheorem} with $L = 0$.  Thus, $\fpt{f _{\mu_p}} \geq 1$ for every $p \in \GG$ and $\mu_p \in \pi^{-1}(p)$, and Remark \ref{RationalityRemark} shows that equality must hold.  If $W_p:= \pi^{-1}(p)$, we see that $A, \LLambda$, and $W_p$ satisfy the hypotheses of Lemma \ref{AlternateLemma}, and so we are done.
\end{proof}

\begin{Theorem} 
\label{VeryGeneralTheorem}
Log canonicity equals dense $F$-purity for every polynomial in $S$ whose coefficients are algebraically independent over $\mathbb{Q}$.
\end{Theorem}

\begin{proof}  We assume that $f$ has support $\sM$ and algebraically independent coefficients $u_1, \cdots, u_n$.  If $A := \Z[u_1, \cdots, u_n]_{ \prod u_i } \subseteq \mathbb{C}$, then $ f \in S_A$ and $\Support(f_{\mu}) = \sM$ for all maximal ideals $\mu \subseteq A$.   Set $\gamma = \min \set{ 1, \fpt{\sM}}$.  By Proposition \ref{Monomialfpt: P},  $\fpt{\sM} = \max \set{ | \s | : \s \in \P}$, and we choose $\ll \in \P$ with $| \ll | = \gamma$.   By Remark \ref{RationalCoordinates: R} , we may assume that $\ll$ has rational coordinates.  If $\GG$ denotes the set of primes $p$ such that $(p-1) \cdot \ll \in \mathbb{N}^n$, then $^{\#} \GG = \infty$ by Dirichlet's theorem on primes in arithmetic progressions.  Fix a prime $p \in \GG$.   

It follows from the monomial theorem that $\x^{(p-1) \EM \ll}$ appears in $f^{(p-1) \gamma}$ with coefficient  \begin{equation} \label{coefficient} 0 \neq \Theta_{\ll,p}(u_1, \cdots, u_n) = \sum_{ \stackrel{| \k | = (p-1) \gamma}{\EM \k = (p-1) \cdot \EM \ll} } \binom{ (p-1) \cdot \alpha}{ \k } \uu^{\k} \in \Z[u_1, \cdots, u_n] \subseteq A. \end{equation} 

As $\gamma \leq 1$, $\binom{ (p-1) \cdot \gamma}{\k} \neq 0 \bmod p$ for each $\k$ in \eqref{coefficient}.  By hypothesis, $\Z[u_1, \cdots, u_n]$ is a polynomial ring, and it follows that $\Theta_{\ll,p}(\uu)$ induces a non-zero element  of the polynomial ring $\Z / p \Z [u_1, \cdots, u_n] \subseteq A/p A$.  Consider the map $\Spec A \stackrel{\pi}{\to} \Spec \Z$ induced by the inclusion $\Z \subseteq A$.  We have just shown that $W_p := D(\Theta_{\ll,p}(\uu)) \cap \pi^{-1}(p)$, is a dense (open) subset of the fiber $\pi^{-1}(p)$.  Let $\mu_p$ be a maximal ideal in $W_p$.  By definition, the image of $\Theta_{\ll,p}(\uu)$ is non-zero in $A / \mu_p$, and \eqref{coefficient} shows that $\x^{(p-1) \EM \ll}$ is contained in $\Support\( \(f_{\mu_p}\)^{(p-1) \gamma} \)$ but not in $\bracket{\m}{}$ (as $\ll \in \P$).  Thus,$\(f_{\mu_p} \)^{(p-1) \gamma} \notin \bracket{\m}{}$, which allows us to apply Lemma \ref{OneConditionSplittingLemma} (and Remark \ref{RationalityRemark}) to $f_{\mu_p} \in S_A(\mu_p)$ to conclude that 
$\fpt{ f_{\mu_p}} = \gamma = \min \set{ 1, \fpt{\sM}}$.  We see that $A, \GG$, and $W_p$ satisfy the conditions of Lemma \ref{AlternateLemma}, and so we are done.
\end{proof}

\begin{Remark}  An important difference between Theorem \ref{UniquePointTheorem} and Theorem \ref{VeryGeneralTheorem} is that $W_p = \pi^{-1}(p)$ in the former, while we only know that $W_p \subseteq \pi^{-1}(p)$ in the latter.  It would be interesting to investigate under what conditions $F$-pure thresholds remain constant over the fibers of certain distinguished primes in $\Spec \Z$.
\end{Remark}

\bibliographystyle{alpha}
\bibliography{refs}

\end{document}